\documentclass[11pt,a4paper]{amsart}

\usepackage{color}
\usepackage{float}
\usepackage{caption}
\usepackage{longtable}
\usepackage{graphicx}
\usepackage{epsfig}
\usepackage{color}
\usepackage{array}
\usepackage{amsmath}
\usepackage{amsthm}
\usepackage{amssymb}
\usepackage{enumerate}
\usepackage{amscd}
\usepackage{epic}
\usepackage[bookmarksnumbered,bookmarksopen,pdfusetitle,unicode]{hyperref}
\numberwithin{equation}{section}
\theoremstyle{plain}
\newtheorem{thm}{Theorem}[section]

\newtheorem{prop}[thm]{Proposition}

\theoremstyle{definition}
\newtheorem{defn}[thm]{Definition}

\theoremstyle{remark}
\newtheorem{rem}[thm]{Remark}

\DeclareGraphicsExtensions{.eps,.ps,.eps.gz,.ps.gz,.eps.Z,.pdf,.png}
\graphicspath{ {./figures/} }

\usepackage{subfig}

\title[The embedded Nash problem of birational model of rational triple singularities]{The embedded Nash problem of birational models of rational triple singularities}

\author{B.Karadeniz}
\address{Department of Mathematics,  Gebze Technical University\\  Gebze 41400, Kocaeli, Turkey}
 \email{busrakaradeniz@gtu.edu.tr }

 \author{H. Mourtada}
\address{Universit\'e de Paris, Sorbonne Universit\'e, CNRS \\
Institut Math\'ematiques de Jussieu-Paris Rive Gauche\\
F-75013 Paris, France\\}
 \email{hussein.mourtada@imj-prg.fr }

 \author{C.Pl\'enat}
\address{Aix Marseille Univ, CNRS, Centrale Marseille, I2M\\
CMI, Technop\^ole Ch\^ateau-Gombert\\
39, rue F. Joliot Curie, 13453 Marseille Cedex 13\\
}
 \email{camille.plenat@univ-amu.fr}
 
\author{M. Tosun}
\address{Department of Mathematics, Galatasaray University\\  Ortak{\"o}y 34357,  Istanbul, Turkey}
\email{mtosun@gsu.edu.tr}
\thanks{This work is supported by the projects "PHC Bosphore no.42613UE", "TUBITAK no.118F320" and 
"Galatasaray University no.16.504.001BAP". The second named author is also partially supported by Project ANR-LISA-17-CE40-0023.}


\begin{document}

\maketitle 

\tableofcontents

\section{Introduction}

\noindent Given a variety $X$ defined over an algebraically closed field of characteristic $0,$  we are often not able to exhibit an explicit resolution of its singularities; on the other hand there are infinitely many resolutions of singularities of $X$ giving extra information which is not intrinsic to the singularity. The need for understanding the information which is common to all the  resolutions of singularities of a given space $X$ led Nash (in \cite{N}) to study the arc space of $X$. See also  \cite{DeF,PS} for more details. This paper follows this line of thoughts. The difference here is that we are interested in the embedded resolutions of singularities of $X\subset \mathbb A^n.$ For this purpose, we replace the arc space $X_\infty$  of $X$ with the jet schemes of $X$: the arc space $X_\infty$ of $X$ is the space of germs of formal curves drawn on $X$. The jet schemes are a family of finite dimensional schemes indexed by integers which approximate the infinite dimensional arc space; for $m\in \textbf{N},$ the $m$-th jet scheme $X_m$  of $X$, can be thought {of} as the space of arcs in the ambient space $\mathbb A^n$ whose "contact" with $X$ is greater or equal to $m+1;$ this gives the intuition  why these schemes should detect information about embedded resolutions of singularities. The main question considered in this paper is: can we construct an embedded resolution of singularities from the jet schemes of $X\subset \mathbb A^n$? More precisely, we ask the following much less optimistic question: 
\begin{center}
\textit{ $(\star)$ Can one construct an embedded resolution of singularities of $X\subset \mathbb A^n$ from the irreducible components of the space $X_m^{Sing}$ of jets  centered at the singular locus of $X\subset \mathbb A^n$?} 
\end{center} 

          \noindent This question is studied in \cite{Mo5,Mo4,LMR,MP}. In \cite{MP}, the authors proved that the irreducible components of the jet schemes centered at the singular locus of a rational double point surface singularity (known also as ''simple singularities" in the literature) give a minimal embedded resolution by a birational toric modification of the ambient space. Equivalently, a certain natural family of the irreducible components of  the jet schemes of $X$ centered at the singular point ${0}$ $X_m^0$ is in bijection with the divisorial valuations whose center is a toric divisor on every toric embedded resolution; this bijection is actually a conceptual correspondence since one can associate with any irreducible component of $X^0_m$ a divisorial valuation centered at the origin of $\mathbb A^n$ (see \cite{ELM}). In general, such a statement is hopeless: indeed, even for an irreducible plane curve singularity (say, for the cusp $\{y^2-x^3=0\}\subset \mathbb A^2$), the irreducible components of the jet schemes centered at the origin give divisorial valuations which do not appear in the minimal embedded resolution of the curve singularity (in that case, the minimal embedded resolution makes sense and is unique). \\

\noindent The answer to $(\star)$ is no in general. Indeed, consider the three-dimensional variety defined by
$$X=\{x^2+y^2+z^2+w^5=0\}\subset \mathbb A^4.$$
It has an isolated singularity at the origin $0$. On the one hand, by a direct computation, we see that the jet schemes $X_m^0$ centered at $0$ are irreducible for every $m\geq 1.$  On the other hand, we have two exceptional (irreducible) divisors that appear on every embedded resolution of the singularity (at least those which correspond to the two essential divisors appearing in the abstract resolution of the origin $0$ ) of $X;$ these are the divisors associated with the monomial valuations on $k[x,y,z,w]$ defined by the vectors $(1,1,1,1)$ and $(2,2,2,1).$ The valuation associated
with the vector $(2,2,2,1)$ does not correspond to any of the schemes  $X_m^0$ with $m\geq1.$ Note that this example is one of the counterexamples to the Nash problem given in \cite{JK}; note also that the Nash correspondence is bijective in dimension $2$ \cite{dedo,bope} but there are many counter-examples in higher dimension (\cite{IK,DeF2}). This suggests that a reasonable frame to study the question $(\star)$ is the surface singularities.\\

\noindent In this paper we study the question $(\star)$ for a family of hypersurface singularities whose normalizations are rational triple point singularities (RTP-singularities, for short).  These hypersurfaces are classified in  \cite{mag} and are called the non-isolated forms of RTP-singularities. We prove that, for this class of singularities, the answer to $(\star)$ is {{positive}}. When $X$ is of that type, we determine again a natural family of irreducible components of $X_m^{Sing}$, $m\geq 1$ whose associated divisorial valuations are monomial, hence defined by some vectors in $\mathbb{N}^3$.   For all of the non-isolated forms of RTP-singularities except when $X$ is of type $B_{k-1,2l-1}$, we show that these vectors give a regular subdivision $\Sigma$ of the dual Newton fan of $X$ and hence a nonsingular toric variety $Z_\Sigma$; since our singularities are Newton non-degenerate \cite{Va,AGS,mag}, this gives a birational toric morphism $Z_\Sigma\longrightarrow \mathbb A^3$ which is an embedded resolution of $X\subset \mathbb A^3$; the irreducible components of the exceptional divisor correspond to the natural set of irreducible components of $X_m^{Sing}.$ When $X$ is of type $B_{k-1,2l-1},$ we again build  a toric embedded resolution  from the irreducible components of the jet schemes  which does not factor through the toric map associated with the dual Newton fan (such resolutions of non-degenerate singularities also appear when one considers an embedded resolution in family; work in progress of Leyton-Alvarez, Mourtada and Spivakovsky). This again shows mysteriously that the jet schemes tell us something about the "minimality" of the embedded resolution, as in the case of rational double point singularities.  \\

 \noindent The paper is organized as follows: Section 2 present a reminder  on RTP-singularities. Section 3 is devoted to jet schemes and how one can associate a divisorial valuation with a component of the jet schemes; it also contains a summary of the approach to the embedded resolutions which will be constructed in the sequel. Each of the remaining sections is  devoted to a class of RTP-singularities (given in the table of contents above): we compute each of the jet schemes and present the results in the jet graph (see Section 3). We then give the toric embedded resolution which comes  from the jet schemes.  We give the explicit computations with all details  for the classes $E_{6,0}$ and $A_{k-1,l-1,m-1}.$ For the other classes, except a subclass of the type $B,$ we proceed similarly, so we present here only the results of the computations. The case $B_{k-1, 2l-1}$ with $k\ge l$ is treated in detail as its behavior is completely different from the other cases. This is related to the fact that the abstract toric resolution of  $B_{k-1, 2l-1}$ which is obtained from a subdivision of the two dimensional cones of the dual Newton fan is not minimal \cite{mag}. \\

\subsection*{Acknowledgements} We are grateful to the two referees for their comments and corrections which greatly
improved the article.  The third author would like  also to thank the UMI fibonacci (Pisa) and the Lama (University of Savoie-Montblanc (Chamb\'ery)) for their hospitality during the preparation of this work.

\section{RTP-singularities}

 \noindent  Let $X$ denote a germ of a surface $(X,0)\subset (\mathbb C^N,0)$ having a singularity at $0$. We say that the singularity of $X$ is rational if 
{ $H^1(\tilde X,\mathcal O_{\tilde X})=0$} where  $\pi :\tilde X\longrightarrow X$ is a resolution of $X$. This definition does not depend on the resolution $\pi $. It is well known that the rational singularities of complex surfaces have nice combinatorial properties which can be computed via their resolutions. In  \cite{Ar}, the rational singularities of multiplicity $3$ are classified by their dual graphs associated with the irreducible components of the minimal resolutions. For short, we call  RTP-singularities this class of rational singularities. They are among the surface singularities defined  in $\mathbb C^4$ and, each of which is defined by three equations given in \cite{tj}. The classification problem of rational singularities of multiplicity $m\geq 3$ is well studied in  \cite{le-tosun} and \cite{stevens2}.
 
 \vskip.2cm

\noindent   In \cite{mag}, the authors obtain the equations of a class of hypersurfaces in $\mathbb C^3$ having nonisolated singularities obtained by projecting the equations of RTP-singularities to a generic hyperplane in $\mathbb C^4$ and, they call them the non-isolated forms of RTP-singularities since the normalizations of these hypersurfaces in $\mathbb C^3$ are exactly the RTP-singularities. They also show that:

\begin{thm}
The RTP-singularities  are  non-degenerate with respect to their Newton polyhedron. In particular, they can be resolved by a toric birational map $Z\longrightarrow \mathbb{C}^4.$ \\
\end{thm}

 \noindent  { In \cite{mag}, the dual graph of the minimal resolution for all RTP-singularities, except  those of type $B_{k-1,2l-1}$ for $k\geq l$ (see Section 6) are constructed by refining the dual Newton fan of the corresponding non-isolated forms of RTP-singularity (see also \cite{O2,Va}). In the  case of the nonisolated form of a rational singularity of type $B_{k-1,2l-1}$ with $k\geq l,$ the resolution obtained by the subdivision of the corresponding dual Newton fan is not minimal}: { c}onsider the vectors $R:=(2l-2,2,2k+1)$, $Q:=(2k-l+2,1,2k-l+2)$, $P:=(l-1,1,l-1)$, $V:=(2k-l, 1, 2k-l+1)$ and $U:=(l-1,1,l)$ coming out in the subdivision of the dual Newton fan of that singularity: \\

\begin{figure}[H]
	\setlength{\unitlength}{0.30cm}
	
	\includegraphics[scale=0.7]{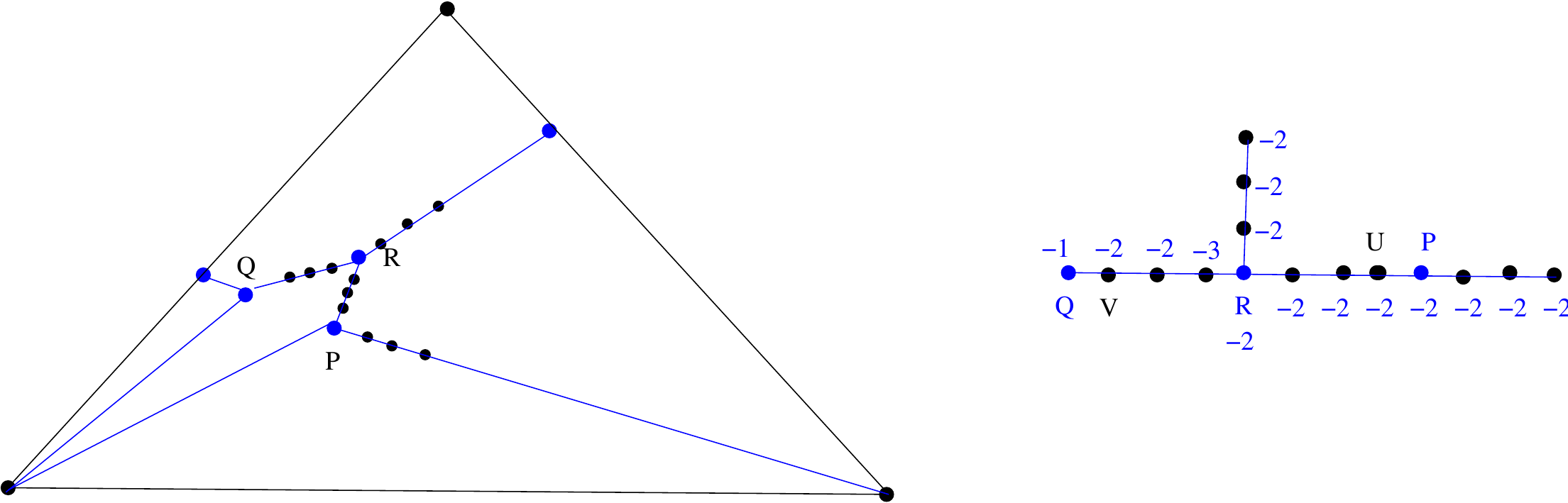}
\caption{Dual Newton fan of a $B_{k-1,2l-1}$ singularity (with $k\geq l$), and its dual (abstract) resolution graph}
\end{figure}

 \noindent Using \cite{O2}, one can compute the self-intersections of the irreducible components of the exceptional divisors corresponding to these vectors; they are given by the number decorating the dual graph given on the right hand side. We omit the genus decorations which are all $0$ in this case.
The exceptional component corresponding to  the vector $Q$ has self-intersection $(-1)$; by Castelnuovo's criterion,  (cf. for example \cite{H}, chapter V), that component can be contracted to a nonsingular point without creating singularities. If we continue to contract each $(-1)$-curve and neighboring components accordingly we obtain a $(-3)$-curve on the segment $[QR]$ and the dual graph of the minimal resolution of the RTP-singularity of type $B_{k-1,2l-1}$, $k\geq l.$

\section{Jet schemes}

\noindent Let $k$ be an algebraically closed field of arbitrary characteristic and $X$ be a $k$-algebraic variety. For $m \in \mathbb{N},$ the jet scheme 
$X_m$ is the scheme representing the functor 
$$F_m \colon
\begin{array}{c}
k\text{-}Schemes \to Sets\\
Spec(A)\mapsto Hom_k\left(Spec \left(A[t]/(t^{m+1})\right),X\right )
\end{array} $$

\noindent where $A$ is a $k$-algebra.  The closed points of $X_m$ are in bijection with the $k[t]/(t^{m+1})$ points of $X$. By definition, we have $X_0=X$. Moreover, for $m,p \in \mathbb{N}$ with $m > p$, we have a  canonical projection $\pi_{m,p}: X_m \longrightarrow X_p$  which is induced by the surjection $A[t]/(t^{m+1}) \longrightarrow A[t]/(t^{p+1}).$ These morphisms are affine and verify $\pi_{m,p}\circ \pi_{q,m}=\pi_{q,p}$
for $p<m<q;$ they define a projective system whose limit is a scheme that we denote $X_{\infty}$ and which is called the arc space of $X$. Let us denote the canonical projection $\pi_{m,0}:X_m\longrightarrow X_0$ by $\pi_{m}$ and, the canonical morphisms $X_{\infty}\longrightarrow X_m$ by $\Psi_m$.   \\

We show here for a surface $X=\{f=0\}\subset k^3$ (since the varieties that we are considering are defined this way) that the functor of the jet schemes is representable; this explains also how one determines jet schemes. We have 
$$X=\mathrm{Spec}\frac{k[x,y,z]}{(f)}.$$
For a $k-$algebra $A,$ an element $\gamma$ in $F_m(\mathrm{Spec}(A))$ corresponds to
a $k-$algebra homomorphism $$\gamma^*: \frac{k[x,y,z]}{(f)}\longrightarrow \frac{A[t]}{(t^{m+1})}.$$
 The data of such a $\gamma$ is equivalent to  the data of 
 $$\gamma^*(x)=x(t)=x_0+x_1t+\cdots+x_mt^m \in A[t]/(t^{m+1}),$$
 $$\gamma^*(y)=y(t)=y_0+y_1t+\cdots+y_mt^m\in A[t]/(t^{m+1}),$$
 $$\gamma^*(z)=z(t)=z_0+z_1t+\cdots+z_mt^m\in A[t]/(t^{m+1});$$
such that $$f(x(t),y(t),z(t))=F_0+F_1t+\cdots+F_mt^m+\cdots=0~~~\mathrm{mod}~~(t^{m+1}).$$
Here, for $i\geq 0, F_i$ is simply the coefficient of $t^i$ in the expanding of $ f(x(t),y(t),z(t)).$\\

Hence, the data of such a $\gamma$ is equivalent to the data of $x_j, y_j, z_j\in A$ with $j=0,\ldots,m$ such that  $F_i(x_0,y_0,z_0,\ldots,x_i,y_i,z_i)=0$ with $i=0,\ldots,m.$ This is equivalent to determining an $A$-point of the scheme 
$$X_m:=\mathrm{Spec}\frac{K[x_i,y_i,z_i;i=0,\ldots,m]}{(F_0,\ldots,F_m)},$$ 
which then represents the functor $F_m$ and, is by definition the $m$-th jet scheme of $X.$ \\

From now on, we assume that $X$ is a surface in $\mathbb{C}^3$ {{ defined by $ \{f(x,y,z)=0\}$ }} and $Y$  is {{a}} subvariety of $X$. {{Let $m\in \mathbb{N}$ }}We denote by $X_m^{Y}:=\pi_{m}^{-1}(Y).$ We consider  a special type of the irreducible components  {{of}} $X_m^{Y}, m\in \mathbb{N}$ where $Y$ is the singular locus of $X$ or $Y\subset X$ is a curve {{contained}} in a coordinate hyperplane of $\mathbb{C}^3$. To such $Y,$ {{we associate}} a divisorial valuation over $\mathbb{C}^3$ with  an irreducible {{component }} $\mathcal{C}_m\subset X_m^{Y}$ {{ in the following way.}} \\

\noindent  Let ${\psi_m^a}: \mathbb{C}^3_\infty \longrightarrow \mathbb{C}^3_m$ be the truncation morphism associated with the ambient space $\mathbb{C}^3$, {{here the exponent $"a"$ stands for \it{ambient map} }}. The morphism ${\psi_m^a}$ is a trivial fibration, hence
 ${\psi_m^a}^{-1}(\mathcal{C}_{m})$ is an irreducible cylinder in $\mathbb{C}^3_\infty.$ Let $\eta$ be the generic point of ${\psi_m^a}^{-1}(\mathcal{C}_{m})$. { By Corollary 2.6 in} \cite{ELM},  the map  $\nu_{\mathcal{C}_m}:\mathbb C[x,y,z]\longrightarrow \mathbb{N}$ defined by  
$$ \nu_{\mathcal{C}_m}(h)=\mbox{ord}_t h\circ \eta $$
{{ is}} a divisorial valuation on $\mathbb C^3.$ To each irreducible component $\mathcal{C}_m$ of $X_m^{Y}$, let us  associate a vector, called the weight vector, in the following way:
$$v(\mathcal{C}_m):=(\nu_{\mathcal{C}_m}(x),\nu_{\mathcal{C}_m}(y),\nu_{\mathcal{C}_m}(z))   \in \mathbb{N}^{3}.$$ 
\noindent Now, we want to characterize {{ the}}  irreducible components of $X_m^{Y}$ that will allow us to construct an embedded resolution of $X$:
For $p\in \mathbb{N},$ we consider the following cylinder in the arc space:
$$Cont^p(f)=\{\gamma \in \mathbb{C}^3_\infty ; \mbox{ord}_tf\circ \gamma=p\}.$$

\begin{defn} {{Let $X:\{f=0\}$ be a surface in $ \mathbb{C}^3$ and let $Y$ be a subvariety of $X$}}.\\
 $(i)$ The elements of the set:
$$EC(X):=\{\text{Irreducible components }\mathcal{C}_m \ \text{of} \ X_m^{Y} \ \text{such that }  {\psi_m^a}^{-1}(\mathcal{C}_{m})\cap Cont^{m+1}{f} \not = \emptyset$$
$$\ \ \ \ \ \ \ \ \ \text{and } v(\mathcal{C}_m)\not=v(\mathcal{C}_{m-1})~~\text{for {any component}} ~~\mathcal{C}_{m-1} ~~\text{verifying}~~ 
\pi_{m,m-1}(\mathcal{C}_m)\subset \mathcal{C}_{m-1}, m\geqslant 1 \}$$
are called the \emph{essential components} for $X$. 

\noindent $(ii)$ the elements of the set of associated valuations: 
$$EV(X):=\{\nu_{\mathcal{C}_m},~~\mathcal{C}_m\in EC(X) \},$$
are called \emph{embedded-valuations} for $X$. 
\end{defn}

\noindent This means that the elements of $EV(X)$  appear in the embedded toric resolution of $X$.  We will be interested in a subset of $EV(X),$ which {{gives}} us an embedded resolution.  In the following sections, in order to determine such a subset when $X$ is a non-isolated form of an RTP-singularity, we will study the $m$-th jet schemes of $X,$ for $m\leq l$ with $l$ large enough. We will encode the structure of these jet schemes by a levelled  graph whose vertices correspond to the irreducible components of $X^Y_m$ for an integer $m$; two vertices at the level $m$ and $m-1$ are joined by an edge if the transition morphism $\pi_{m,m-1}$ {{sends}} the corresponding components one into the other \cite{Mo3}. An element of $EV(X)$ corresponding to a component $\mathcal{C}_m \in EC(X)$ is actually a monomial (or toric) valuation (see proposition 2.3 in \cite{MP}) and is 
defined by the vector $v(\mathcal{C}_m)=(a,b,c)$: this means that, for $h=  \sum_{\{(i,j,k)\}}a_{(i,j,k)}x^iy^jz^k\in \mathbb{C}[x,y,z]$ we have:
$$\nu_{\mathcal{C}_m}(h)=\mathrm{min}_{\{(i,j,k)\mid a_{(i,j,k)}\not=0\}}\{ai+bj+ck\}.$$
By subdividing the first quadrant of $\mathbb{R}^3$ using the vectors $v(\mathcal{C}_m)$ for some $\mathcal{C}_m \in EC(X),$ we obtain a fan $\Sigma$ whose support is the first quadrant of $\mathbb{R}^3$ and whose one dimensional cones are generated by these $v(\mathcal{C}_m)$'s. Note that one can obtain different fans from a set of vectors in $\mathbb{R}^3$, depending on the way one relies the vertices and, some of them may not be regular, but here we are interested in finding a regular fan. Hence we have a proper birational map $\mu_\Sigma:Z_\Sigma\longrightarrow \mathbb{C}^3$ where $Z_\Sigma$ is smooth and the irreducible components of the exceptional divisor of $\mu_{\Sigma}$ correspond to the vectors $v(\mathcal{C}_m)$ that we consider. More precisely, the divisorial valuations corresponding to the irreducible components of the exceptional divisor of $\mu_\Sigma$ are exactly the $\nu_{\mathcal{C}_m}$ associated with the components   $\mathcal{C}_m$ that we consider. We will find such a regular fan $\Sigma$ for a non-isolated {{form  $X$ of}} an RTP-singularity which is not of the type $B_{k-1,2l-1}$ ($i.e$ we will construct  $\Sigma$ using the vectors of type  $v(\mathcal{C}_m))$ that refines the dual Newton fan of $X\subset \mathbb{C}^3$. Thanks to Varchenko's theorem \cite{Va}, this gives that $\mu_\Sigma$ is an embedded resolution of $X\subset \mathbb{C}^3.$ On the other hand, for a $B_{k-1,2l-1}$-singularity, we cannot apply Varchenko's theorem because there is no $\Sigma $ refining the dual Newton fan as described above; nevertheless we build a regular fan $\Sigma$ satisfying the properties above and, we prove by studying the total transform of our singularity by $\mu_\Sigma$ that $\mu_\Sigma:Z_\Sigma\longrightarrow \mathbb{C}^3$ is an embedded resolution of the $B_{k-1,2l-1}$ singularity.
  
\section{RTP-singularities of type $E_{6,0}$}

\noindent  The singularity of $X\subset \mathbb{C}^3$ defined by  the equation:
$$f(x,y,z)=z^{3}+y^{3}z+x^{2}y^{2}=0$$ is called $E_{6,0}$-type singularity. Its dual Newton fan is given in Figure \ref{dualE60}. \\

 In this section, we compute explicitly the $m$-th jet schemes (for $m\leq 18$) and we determine a subset of $EV(X)$ which gives a regular subdivision of the dual Newton fan as explained in the previous section. We represent the irreducible components as a graph in Figure \ref{E60}, where we also weight the vertex associated with a component $\mathcal{C}_m$ by the vector $v(\mathcal{C}_m)$ also defined in the previous section. For a component $\mathcal{C}_m$ which projects by the maps $\pi_{m,m-1}$ given in Section $3$ on a monomial component ($i.e.$ a component whose associated valuation is monomial)  $\mathcal{C}_{m-1},$ which is not itself monomial; we also weight the associated vertex by the unique non-monomial equation which, together with the hyperplane coordinates $\mathcal{C}_{m-1},$ defines $\mathcal{C}_{m}$. That helps for the computations of the irreducible components in the process. Here we do not pay much attention to the edges since they are not relevant for the problem at hand. The arrows in Figure $\ref{E60}$ correspond to a component $\mathcal{C}_m$ such that the inverse image of a dense open set in it gives an irreducible component for every $n\geq m.$ 
First let us fix some notations:

\begin{figure}[!ht]
\includegraphics [scale=0.3,height=6cm,width=12cm]{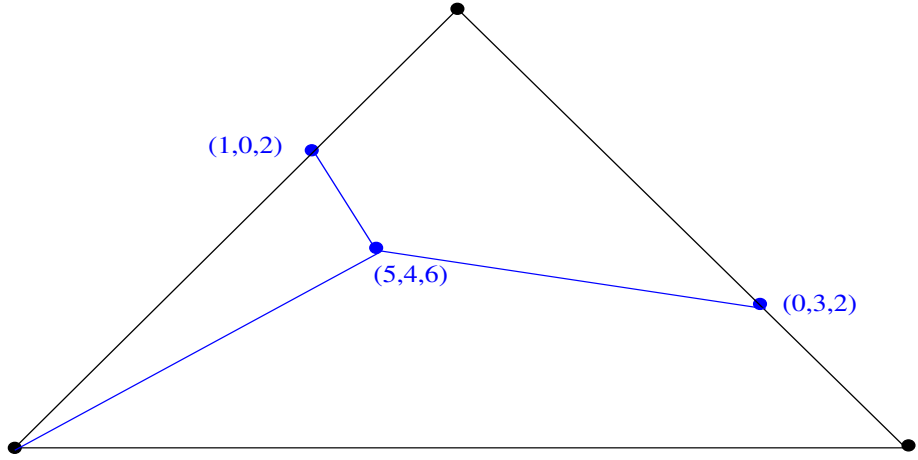}
\caption{Dual Newton fan of $E_{6,0}$ singularity}
\label{dualE60}
\end{figure} 
\vskip.2cm

\noindent  {\bf Notation:} Let
$$(\ast )\  \  f(\sum\limits_{i=0}^{m}x_{i}t^{i}, \sum\limits_{i=0}^{m}y_{i}t^{i}, \sum\limits_{i=0}^{m}z_{i}t^{i} )= \sum\limits_{i=0}^{i=m}F_{i}t^{i} \ \ \ \ mod (t^{m+1})$$
\noindent We know that (e.g. \cite{MP}) the $m$-th jet scheme $X_{m}$ is defined by the ideal 
$$I_{m}=(F_{0}, F_{1},\dots ,F_{m}) \subset \mathbb{C}[x_i,y_i,z_i;i=0,\ldots,m].$$
\subsection{Jet Schemes of $E_{6,0}$}

\noindent For $m\geq 1,$ we will determine the irreducible components of the space of $m-$jets that projects on the singular locus of $X,$ i.e. the irreducible components of $X_m^{Sing}:=\pi_{m,0}^{-1}(V(y_0,z_0))\subset X_m\subset \mathrm{Spec}(\mathbb{C}[x_i, y_i, z_i;i=0,\ldots,m])=\mathbb{C}_m^3;$ here $V(I)$ denotes the variety defined by an ideal $I$ and $\mathbb{C}_m^3$ is the $m$-th jet scheme of the affine three dimensional space $\mathbb{C}^3;$ we insist here that when considering $X_m^{Sing}$ for a given $m,$ the symbol $V(I)$ designates the variety defined by an ideal $I$ in $\mathbb{C}_m^3.$ Recall that $\pi_{m_0}:X_m\longrightarrow X_0=X.$ We also insist on the fact that we consider only the reduced structure of these schemes.\\

\noindent {\bf For $m=1$}, we have $X_1^{Sing}=V(y_0,z_0)\subset\mathrm{Spec}(\mathbb{C}[x_i,y_i,z_i;i=0,1])$ because, if we put $y_0=z_0=0$ in the equation $(\ast)$  we get $F_0=F_1=0$ modulo the ideal $(y_0,z_0).$ Hence, $X_1^{Sing}$ consists of a unique irreducible component, denoted by $\mathcal{C}_{1,1}.$ The weight vector of $\mathcal{C}_{1,1}$ is $(0,1,1).$\\
 
\noindent {\bf For $m=2$}, we have $X_2^{Sing}=\pi_{2,1}^{-1}(\mathcal{C}_{1,1});$ this uses the fact $\pi_{2,0}= \pi_{1,0}\circ \pi_{2,1}.$ A direct computation using the equation $(\ast)$ gives:
$$F_2=x_0^2y_1^2~~~\mathrm{mod}~~ (y_0,z_0).$$ 
Hence $X_2^{Sing}=V(y_0,z_0,x_0^2y_1^2)\subset \mathrm{Spec}(\mathbb{C}[x_i, y_i, z_i;i=0,1,2])=\mathbb{C}_2^3$. We deduce that  $X_2^{Sing}$ has two irreducible components 
$\mathcal{C}_{2,1}:=V(y_0, z_0, x_0)$ and $\mathcal{C}_{2,2}:=V(y_0, z_0, y_1)$ both are sent via $\pi_{2,1}$ into $\mathcal{C}_{1,1};$ there weight vectors are respectively $(1,1,1)$ and $(0,2,1).$ These vectors are represented in Figure \ref{E60} at the levels $m=1$ and $m=2.$\\

\noindent {\bf For $m=3$}, using  the fact $\pi_{3,0}=\pi_{2,0}\circ \pi_{3,2},$ it is sufficient to study $\pi_{3,2}^{-1}(\mathcal{C}_{2,j})$ with $j=1,2$ to understand $X_3^{Sing}.$  
\begin{itemize}
\item To find $\pi_{3,2}^{-1}(\mathcal{C}_{2,1}),$ we compute $F_3$ modulo the ideal $(x_0,y_0,z_0)$ and we obtain: 
$$F_3= z_1^3~~~\mathrm{mod} (x_0,y_0,z_0);$$  
Hence we obtain that $\mathcal{C}_{3,1}:=\pi_{3,2}^{-1}(\mathcal{C}_{2,1})=V(x_0,y_0,z_0,z_1)$ is irreducible. 
\item Similarly, we obtain that   $\mathcal{C}_{3,2}:=\pi_{3,2}^{-1}(\mathcal{C}_{2,2})=V(y_0,y_1,z_0,z_1)$ is irreducible.
\end{itemize}
\noindent So we have $X_3^{Sing}=\mathcal{C}_{3,1}\cup \mathcal{C}_{3,2}$ where $\mathcal{C}_{3,1}$ and  $\mathcal{C}_{3,2}$ are both irreducible and clearly there is no inclusions between them: indeed, $\mathcal{C}_{3,1}$ is included in $V(x_0)$ but  $\mathcal{C}_{3,2}$  is not and,  $\mathcal{C}_{3,2}$ is included in $V(y_1)$ but  $\mathcal{C}_{3,1}$  is not. We conclude that $\mathcal{C}_{3,1}$ and  $\mathcal{C}_{3,2}$ are the irreducible components of $X_3^{Sing}$. Their associated weight vectors are respectively $(1,1,2)$ and $(0,2,2).$\\

\noindent {\bf For $m=4$}, as in the previous case, it is sufficient to consider $\pi_{4,3}^{-1}(\mathcal{C}_{3,j}),$ with $j=1,2.$ As the computations go almost in the same way, we just announce what we obtain:
\begin{itemize}
	\item  To determine $\pi_{4,3}^{-1}(\mathcal{C}_{3,1})$ we compute $F_4$ modulo the ideal $(x_0,y_0,z_0,z_1).$ We have 
	$$F_4=x_1^2y_1^2~~~\mathrm{mod}~~ (x_0,y_0,z_0,z_1).$$ Hence, $\pi_{4,3}^{-1}(\mathcal{C}_{3,1})$ has $2$ irreducible components  $\mathcal{C}_{4,1}=V(x_0,y_0,y_1,z_0,z_1)$  and $\mathcal{C}_{4,2}=V(x_0,y_0,x_1,z_0,z_1).$ 
	\item Similarly we have  $\pi_{4,3}^{-1}(\mathcal{C}_{3,2})=\mathcal{C}_{4,1}\cup \mathcal{C}_{4,3}$ where $\mathcal{C}_{4,3}= V(y_0,y_1,y_2,z_1,z_0).$ 
\end{itemize}
\noindent Then we get $$X_4^{Sing}=\mathcal{C}_{4,1}\cup \mathcal{C}_{4,2}\cup \mathcal{C}_{4,3}$$ which is a decomposition into irreducible varieties. Using a similar argument as in the case of $m=3$, we conclude that there are no mutual inclusions between these components; hence this is the decomposition into irreducible components. The corresponding weight vectors of $\mathcal{C}_{4,1}, \mathcal{C}_{4,2}$ and  $\mathcal{C}_{4,3}$ are 
respectively $(1,2,2),(2,1,2)$ and $(0,3,2).$ Figure \ref{E60} encodes also this information.\\

\noindent {\bf For $m=5$}, we have 
$$X_5^{Sing}=\pi_{5,4}^{-1}(\mathcal{C}_{4,1})\cup \pi_{5,4}^{-1}(\mathcal{C}_{4,2})\cup\pi_{5,4}^{-1}(\mathcal{C}_{4,3})\subset \mathbb{C}_5^3. $$ 
\begin{itemize}
	\item To determine $\pi_{5,4}^{-1}(\mathcal{C}_{4,1}),$ we compute $F_5$ modulo  $(x_0,y_0,y_1,z_0,z_1)$ that we find to be $0.$ We deduce, that ${C}_{5,1}:= \pi_{5,4}^{-1}(\mathcal{C}_{4,1})= V(x_0,y_0,y_1,z_0,z_1)$ is irreducible. A small attention here is needed: The varieties $\mathcal{C}_{4,1}$ and  $\mathcal{C}_{5,1}$ are not the same; they are defined by the same equations but in different rings; they actually define the same valuation on $\mathbb{C}^3$ (see proposition 2.3 in \cite{MP}).
	\item Computing $F_5$ modulo the ideal  $(x_0,x_1,y_0,z_0,z_1)$, we find $F_5=y_1^3z_2=0$. So $\pi_{5,4}^{-1}(\mathcal{C}_{4,2})$ is the union of  $V(x_0,x_1,y_0,y_1,z_0,z_1)$  and $\mathcal{C}_{5,2}:=V(x_0,x_1,y_0,z_1,z_0,z_2).$ 
	\item As for $\pi_{5,4}^{-1}(\mathcal{C}_{4,1}),$ computing $F_5 $ modulo $(y_0,y_1,y_2,z_0,z_1)$ we find zero. This gives that $\mathcal{C}_{5,3}:=\pi_{5,4}^{-1}(\mathcal{C}_{4,3})=V(y_0,y_1,z_0,z_1,z_2)$is irreducible.
\end{itemize}
Hence we obtain $$X_5^{Sing}=\mathcal{C}_{5,1}\cup \mathcal{C}_{5,2}\cup V(x_0,x_1,y_0,y_1,z_0,z_1)\cup \mathcal{C}_{5,3}.$$
Since $V(x_0,x_1,y_0,y_1,z_0,z_1)$ is included in $\mathcal{C}_{5,1}$, the decomposition $$X_5^{Sing}=\mathcal{C}_{5,1}\cup \mathcal{C}_{5,2}\cup  \mathcal{C}_{5,3}$$  
 is the decomposition into the irreducible components. Moreover, the weight vectors of
$\mathcal{C}_{5,j}$ for $j=1,2,3$ are $(1,2,2), (2,1,3)$ and $(0,3,2)$ respectively.\\

\noindent {\bf For $m=6$}, we have $$X_6^{Sing}=\pi_{6,5}^{-1}(\mathcal{C}_{5,1})\cup \pi_{6,5}^{-1}(\mathcal{C}_{5,2})\cup\pi_{6,5}^{-1}(\mathcal{C}_{5,3})\subset \mathbb{C}_6^3. $$ 
\begin{itemize}
	\item To determine $\pi_{6,5}^{-1}(\mathcal{C}_{5,1}),$ we compute $F_6$ modulo the ideal $(x_0,y_0,y_1,z_0,z_1)$ and we find 
	$$\mathcal{C}_{6,1}:=\pi_{6,5}^{-1}(\mathcal{C}_{5,1})=V(x_0,y_0,y_1,z_0,z_1,z_2^3+x_1^2y_2^2)\subset  \mathbb{C}_6^3.$$ Notice that
	$\mathcal{C}_{6,1}$ is isomorphic to the product of an affine space and the hypersurface defined by $\{z_2^3+x_1^2y_2^2=0\};$ this hypersurface is a Hirzebruch-Jung singularity which is well known to be an irreducible quasi-ordinary singularity \cite{CM}; in particular $\mathcal{C}_{6,1}$ is irreducible. Actually, we will see that $\mathcal{C}_{6,1}$ will give rise to an irreducible component of $X_6^{Sing}$ whose weight vector is same as the weight vector associated with $\mathcal{C}_{5,1},$ so it is not an essential component (see definition above): the divisorial valuation associated with it is not monomial while a divisorial valuation associated with an essential component is monomial. Before we continue to study on $X_6^{Sing},$ let us consider $\pi_{m,6}^{-1}(\mathcal{C}_{6,1})$ for $m\geq 7$: For this, we will stratify $\mathcal{C}_{6,1}$ into its regular locus and its singular locus which are defined respectively by 
$x_1=z_2=0$ and $y_2=z_2=0.$ The inverse images $\pi_{7,6}^{-1}(\mathcal{C}_{6,1}\cap \{x_1=z_2=0\})$
and $\pi_{7,6}^{-1}(\mathcal{C}_{6,1}\cap \{y_2=z_2=0\})$ will give the irreducible components of $X_7^{Sing}$ looking like the irreducible components that we have studied before which are the essential components, so give the new weight vectors. The inverse image of the regular part of $\mathcal{C}_{6,1}$ with respect to $\pi_{m,6},$ with $m\geq 7$  is equal to $\pi_{m,6}^{-1}(\mathcal{C}_{6,1}\cap \{z_2\not=0\});$ this latter is defined in $\mathbb{C}_m^3 \cap \{z_2\not=0\}$ by the ideal generated by $x_0,y_0,y_1,z_0,z_1,z_2^3+x_1^2y_2^2$ and
  $$F_j=c_jz_3z_{j-3}+H_j(x_1,\ldots,x_{j-5},y_2,\ldots,y_{j-4},z_3,\ldots,z_{j-3}), c_j\in \mathbb{C}^*$$
  for $7\leq j\leq m $. The functions $F_j$ are linear as we can invert $c_jz_3\not=0.$ Then       
the Zariski closure $\overline{\pi_{m,6}^{-1}(\mathcal{C}_{6,1}\cap \{z_2\not=0\})}$ is irreducible and, is actually an irreducible component of $X_m^{Sing}$ for every $m\leq 7$. Note that the weight vector of $\overline{\pi_{m,6}^{-1}(\mathcal{C}_{6,1}\cap \{z_2\not=0\})}$ is $(1,2,2)$ which is same as the one for $\mathcal{C}_{6,1}$ and for $\mathcal{C}_{5,1};$ hence they don't give an essential component. They are encoded in Figure \ref{E60} by the dashed arrow which starts at the vertex weighted by the 
vector $(1,2,2)$ and the equation $z_2^3+x_1^2y_2^2=0.$ \\
	 
	\item To determine $\pi_{6,5}^{-1}(\mathcal{C}_{5,2}),$ we compute $F_6$ modulo the ideal $(x_0,x_1,y_0,z_0,z_1,z_2)$ and we find that $F_6=y_1^2(z_2y_1+x_2^2).$ So $\pi_{6,5}^{-1}(\mathcal{C}_{5,2}) $ is the union of $\mathcal{C}_{6,2}:=V(x_0,x_1,y_0,y_1,z_0,z_1,z_2)$  and $\mathcal{C}_{6,3}:=V(x_0,x_1,y_0,z_1,z_0,z_2,z_3y_1+x_2^2)$ which are both irreducible. We note that $\pi_{m,6}^{-1}(\mathcal{C}_{6,3})$ is irreducible for every $m\ge 7$ and gives rise to an irreducible component of $X_m^{Sing}$ for every $m\geq 7.$ The irreducibility of the inverse image results from the 
	fact that $\mathcal{C}_{6,3}$ is the product of an affine space and an $A_1$-singularity and the jet schemes of such singularity are irreducible 
	\cite{Mu,Tor} (what applies here for $A_1$ is also true for any rational singularity). The components of $\pi_{m,6}^{-1}(\mathcal{C}_{6,3})$ are not the essential components, they are associated with non-monomial valuations and they have the same weight vector, namely $(2,1,3).$ They are encoded in  Figure \ref{E60} (to the most right of the graph) by the dashed arrow which starts at the vertex weighted by the 
vector $(2,1,3)$ and the equation $x_2^2+z_3y_1=0.$ \\
	
	\item To determine $\pi_{6,5}^{-1}(\mathcal{C}_{5,3}),$ we compute $F_6$ modulo the ideal  $(y_0,y_1,y_2,z_0,z_1)$ and we find that $F_6=z_2^3+x_0^2y_3^3.$  Hence
	$\mathcal{C}_{6,4}:=\pi_{6,5}^{-1}(\mathcal{C}_{5,3})=V(y_0,y_1,y_2,z_0,z_1,z_2^3+x_0^2y_3^2)$ is irreducible. 
By the same argument as in the case of $\pi_{m,6}^{-1}(\mathcal{C}_{6,1}),$ the inverse images $\pi_{7,6}^{-1}(\mathcal{C}_{6,4}\cap \{x_0=z_2=0\})$
and $\pi_{7,6}^{-1}(\mathcal{C}_{6,4}\cap \{y_3=z_2=0\})$ will give rise to the irreducible components of $X_7^{Sing};$ the Zariski closure $\overline{\pi_{m,6}^{-1}(\mathcal{C}_{6,4}\cap \{z_2\not=0\})}$ is irreducible and is actually an irreducible component of $X_m^{Sing}$ for every $m\geq 7.$ This is encoded in Figure \ref{E60} by the dashed arrow starting at the vertex weighted by the 
vector $(0,3,2)$ and the equation $z_2^3+x_0^2y_3^2.$
	\end{itemize}

\noindent To summarize, we obtain $X_6^{Sing}=\mathcal{C}_{6,1}\cup \mathcal{C}_{6,2}\cup  \mathcal{C}_{6,3}\cup \mathcal{C}_{6,4}$ where each $\mathcal{C}_{6,j}$ for $j=1\ldots,4$ is irreducible. Obviously, $\mathcal{C}_{6,2}\subset \mathcal{C}_{6,1} $ and, using the same argument as in the case of $m=3$,  we verify that there is no inclusion among the remaining $ \mathcal{C}_{6,j}$'s. Hence we get the irreducible decomposition
$$X_6^{Sing}=\mathcal{C}_{6,1}\cup  \mathcal{C}_{6,3}\cup \mathcal{C}_{6,4}$$ with the respective weight vectors $(1,2,2),(2,1,3)$ and $(0,3,2)$.\\

\noindent {\bf For $m=7$}, by the above discussions, we have a stratification 
$$X_7^{Sing}= \pi_{7,6}^{-1}(\mathcal{C}_{6,1}\cap \{x_1=z_2=0\})
\cup\pi_{7,6}^{-1}(\mathcal{C}_{6,1}\cap \{y_2=z_2=0\}) \cup \overline{\pi_{7,6}^{-1}(\mathcal{C}_{6,1}\cap \{z_2\not=0\})}\cup $$ 
$$\pi_{7,6}^{-1}(\mathcal{C}_{6,3})\cup \pi_{7,6}^{-1}(\mathcal{C}_{6,4}\cap \{y_3=z_2=0\}) \cup \overline{\pi_{m,6}^{-1}(\mathcal{C}_{6,4}\cap \{z_2\not=0\})}$$ which is the decomposition into irreducible components; indeed, on the one hand using the same argument
as for $m=3$,  there is no inclusions between $\pi_{7,6}^{-1}(\mathcal{C}_{6,3})$ and the other components; on the other hand, the other components are clearly not equal, this means that there are no inclusions between them because they are irreducible and they have the same dimension (actually codimension $7$ in $\mathbb{C}^3_7$). Note that the codimension is easy to compute since the equations are either hyperplane coordinates in $\mathbb{C}^3_7$ or we consider the closure of a constructible set which is defined by hyperplane coordinates and by linear equations. The weight vectors are respectively $(2,2,3),(1,3,3),(1,2,2),(2,1,3),(0,4,3)$ and $(0,3,3).$ Moreover we have 
$\pi_{7,6}^{-1}(\mathcal{C}_{6,4}\cap \{x_0=z_2=0\})=\pi_{7,6}^{-1}(\mathcal{C}_{6,1}\cap \{y_2=z_2=0\}).$
We should also note that although $\mathcal{C}_{6,2}$ is not an irreducible component, its inverse image $\pi_{7,6}^{-1}(\mathcal{C}_{6,2})$ which is equal to $\pi_{7,6}^{-1}(\mathcal{C}_{6,1}\cap \{y_2=z_2=0\})$ gives an irreducible component.\\

We have gone through the arguments which allow to determine all the irreducible components of $X_m^{Sing}$ for $m\leq 18.$ This is encoded in Figure \ref{E60}. Note that $18$ is the quasi-degree of the weighted homogeneous polynomial defining our singularity.\\
 One last important thing is that the axis $Y=\{x=z=0\}$ is  drawn on our singularity. We determine the essential components of $X_m^Y,m\geq 0,$ we find $V(x_0,z_0)\subset X_0\subset \mathbb{C}^3_0$ and $V(x_0,z_0,z_1)\subset X_1\subset \mathbb{C}^3_1$ whose weight vectors are respectively $(1,0,1), (1,0,2).$\\
 
To conclude, the essential components are the irreducible components of $X_m^Z$ (where $Z$ is the singular locus of $X$ or $Z=Y$ is the $y$-axis)  whose defining equations are hyperplane coordinates and, their associated valuations are monomial and determined with their weight vectors.  Hence we get the graph in figure \ref{E60} for the jet schemes.\\

\begin{prop}\label{vE60} 
For an $E_{6,0}$-singularity, the monomial valuations associated with the vectors $(0,1,1),(0,2,1),(1,1,1),(0,3,2),(1,1,2),(1,2,2), (2,1,2),(2,1,3),(2,2,3),(3,2,3),(3,2,4),\break (3,3,4),(4,3,5),(5,4,6)$ belong to $EV(X)$.
\end{prop}
\begin{figure}[H]
\includegraphics[height=13cm,width=14cm]{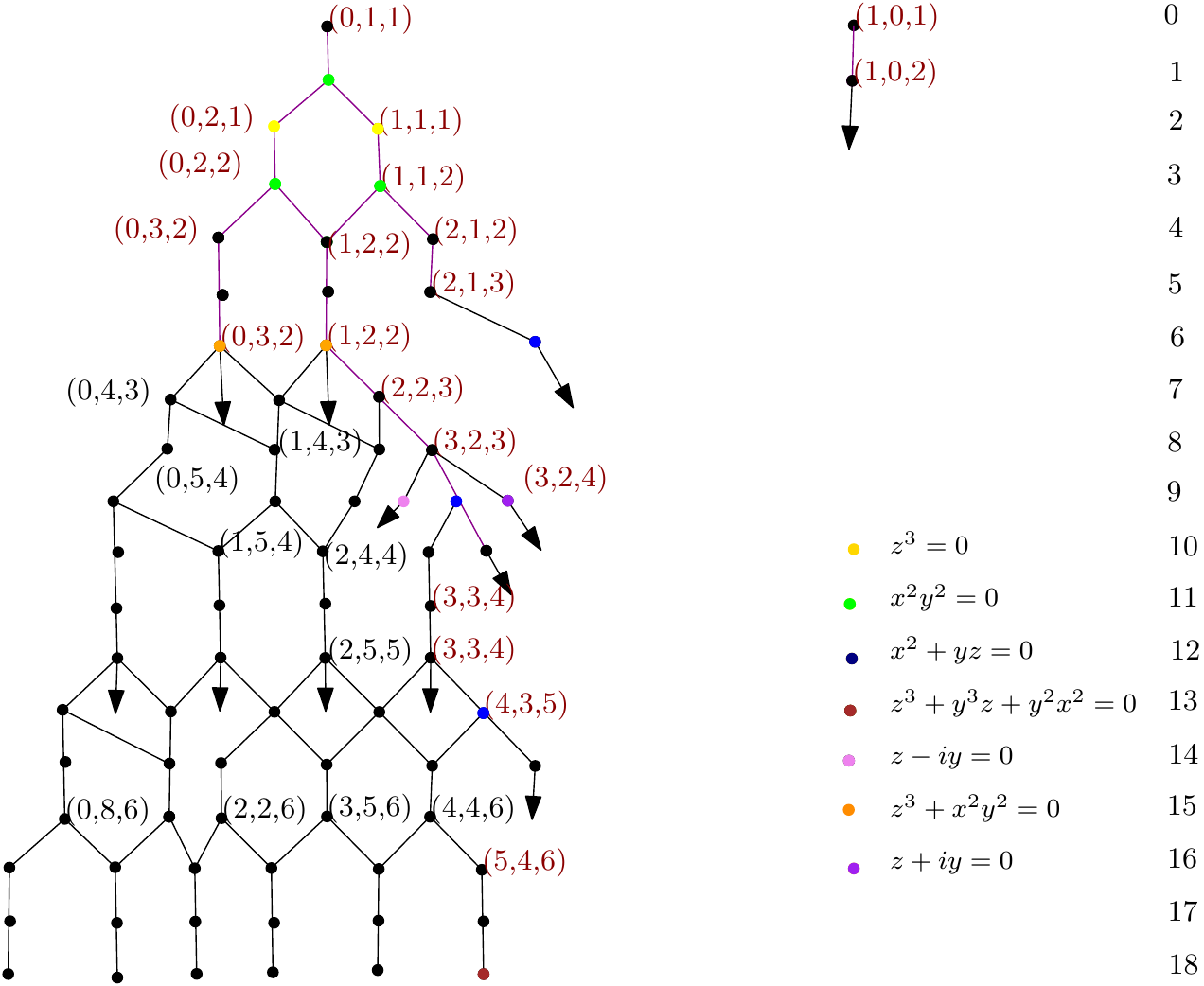}
\caption{Jets schemes of $E_{6,0}$}
\label{E60}
\end{figure}

\subsection{Toric Embedded Resolution of $E_{6,0}$}

\noindent Now we are ready to announce the main result for the surface $X$ of type $E_{6,0}$-singularity.
\begin{thm} There exists a toric birational map $\mu_\Sigma:Z_\Sigma\longrightarrow \mathbb{C}^3$ which is an embedded resolution of $X\subset \mathbb{C}^3$ such that the components of the exceptional divisor of $\mu_\Sigma$ correspond to the irreducible components of the $m$-th jet schemes of $X$ (centered at the singular locus and the intersection of $X$ with the coordinate hyperplane). Moreover this yields a construction (not canonical) of $\mu_\Sigma.$    
\end{thm}

\begin{proof} By  \cite{Va,O2, mag} (see also \cite{MP} for a summary),  an embedded resolution of $X\subset \mathbb{C}^3$ can be obtained by constructing a regular subdivision of the dual Newton fan of $X\subset \mathbb{C}^3$. The dual Newton fan  $\Sigma$ for $E_{6,0}$ is presented in Figure \ref{dualE60} .
 \begin{figure}[!ht]
	\includegraphics[scale=0.5,height=7cm,width=12cm]{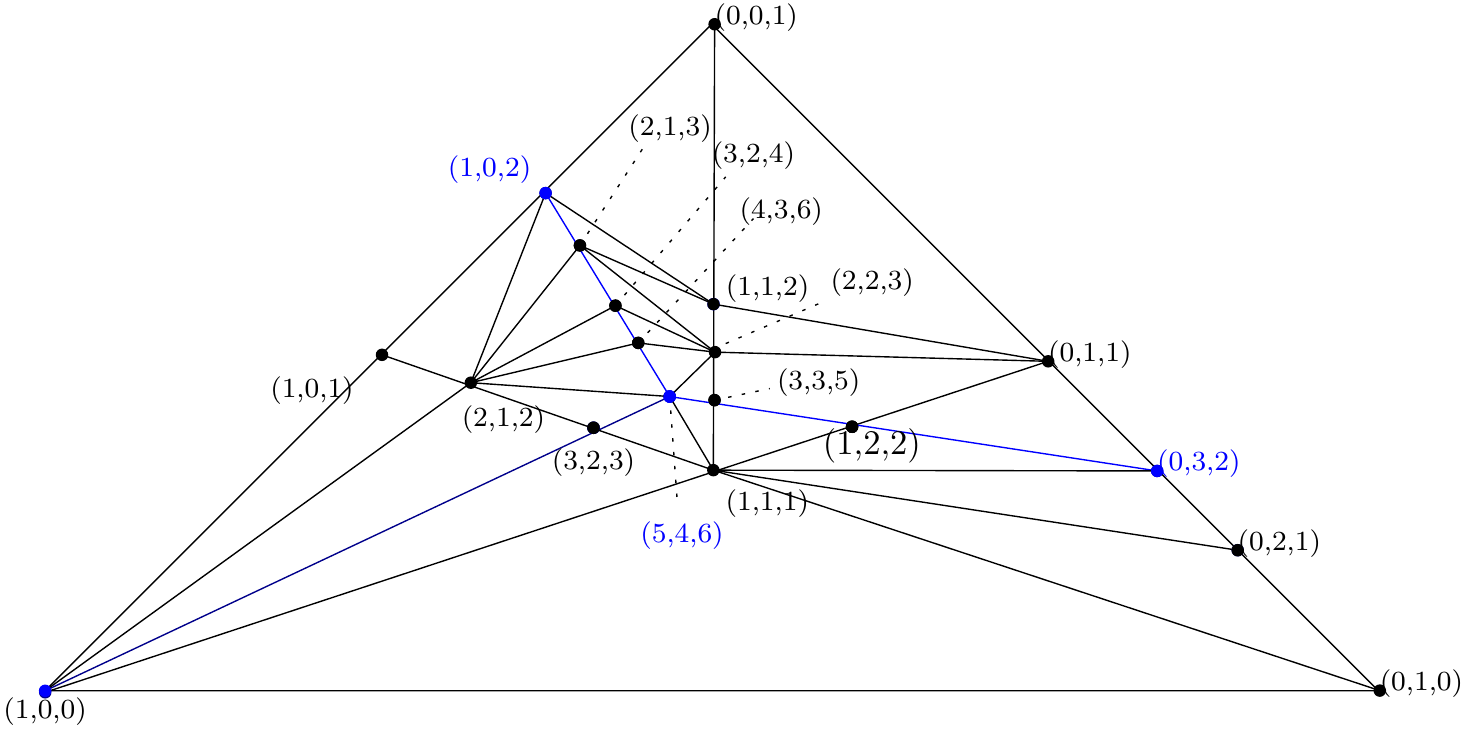}
	\caption{An embedded resolution of $E_{6,0}$}
	\label{resE60}
\end{figure}

\noindent  In Figure $4$, we give a regular subdivision $\Sigma$ where the rays (cones of dimension 1) are the lines supported by the vectors given in proposition \ref{vE60}. To see that this is a regular subdivision, it is sufficient to show that each cone is regular (means that the determinant of the matrix whose columns are any three vectors generating a cone of $\Sigma$ equals $1$). Moreover the 1-dimensional cones (rays) are in bijective correspondence with the components of the exceptional divisors. 

\end{proof}

\section{RTP-singularities of type $A_{k-1,l-1,m-1}$}

\noindent   The singularity of $X\subset \mathbb{C}^3$ defined by  the equations:
    \begin{itemize}
    \item $k\geq \ell \geq m$,
    $$z^3+xz^2-(x+y^k+y^\ell+y^m)y^kz+y^{2k+\ell}=0,$$
    \item $k=\ell < m$,
    $$z^3+(x-y^k)z^2-(x+y^{k}+y^m)y^kz+y^{2k+m}=0.$$
        \end{itemize}
is called $A_{k-1,l-1,m-1}$-type singularity where  $k,\ell,m\geq 1$.

\subsection{Jet Schemes and toric Embedded resolution of $A_{k-1,l-1,m-1}$ when $k=l\le m$}

\noindent The singular locus is $\{y=z=0\}$. So we compute the jets schemes over $\{y=z=0\}$.
 The graph representing the irreducible components of the  jet schemes of $A_{k-1,l-1,m-1}$ is on the  figure \ref{jetA}

\begin{thm}
With preceding notation, the monomial valuations associated with the vectors  
\begin{itemize}
\item $(0,1,1), (0,1,2), \dots (0,1,k+m)$
\item $(s,1,s),\dots,(s,1,m+k-s)$ $ 1\le s \le k-1$
\item $(k,1,k),\dots,(k,1,m)$
\end{itemize}
belong to $EV(X)$. Moreover, these vectors give a toric birational map $\mu_\Sigma:Z_\Sigma\longrightarrow \mathbb{C}^3$ which is an embedded resolution of $X\subset \mathbb{C}^3$ (in the neighborhood of the origin) such that the components of the exceptional divisor of $\mu_\Sigma$ correspond to the monomial valuations defined by them; hence they correspond to  the irreducible components of the $m$-th jet schemes of $X$ (centered at the singular locus and the intersection of $X$ with the coordinate hyperplanes).    
\end{thm}

\begin{figure}[H]
\setlength{\unitlength}{0.5cm}
\includegraphics[scale=0.8,height=16cm,width=16cm]{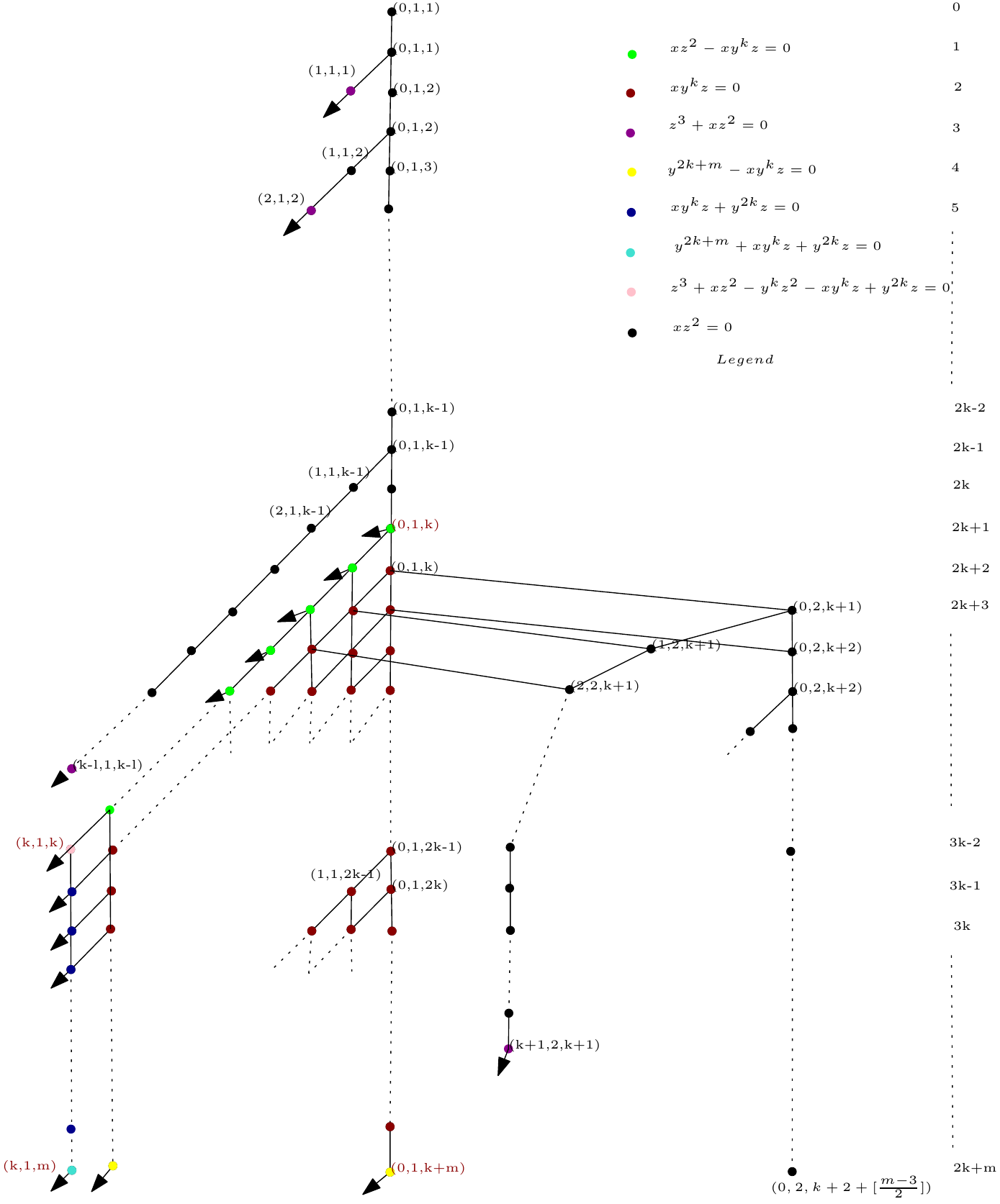}
\caption{ Jets schemes of $A_{k-1,k-1,m-1}$ }
\label{jetA}
\end{figure}

\begin{proof}
The first part of the theorem results from the jet graph. Before showing that the given vectors give a simplicial regular decomposition of the dual Newton fan of $A_{k-1,k-1,m-1}$, let us study their positions in the fan:
\begin{itemize}
\item for all $0\le s \le k$, we have $(s,1,k)\in [(0,1,k),(k,1,k)]$
\item for all $0\le s \le k$,  we have $(s,1,k+m-s)\in [(0,1,k+m),(k,1,m)]$: \\
$$\left \vert \begin{array}{ccc}
\  k&0&s\\
\ 1&1&1\\
\ m&k+m&k+m-s\\
\end{array}\right \vert=\left \vert \begin{array}{ccc}
\  k&0&s\\
\ 0&1&1\\
\ -k&k+m&-s\\
\end{array}\right \vert=0$$
\item the vectors $(\alpha,1,l+\alpha+1)$ for all $0\le \alpha \le k$ are aligned, for each $0\le l\le k$.
\end{itemize}

\begin{figure}[H]
\setlength{\unitlength}{0.5cm}
\includegraphics[scale=0.5,height=5cm,width=16cm]{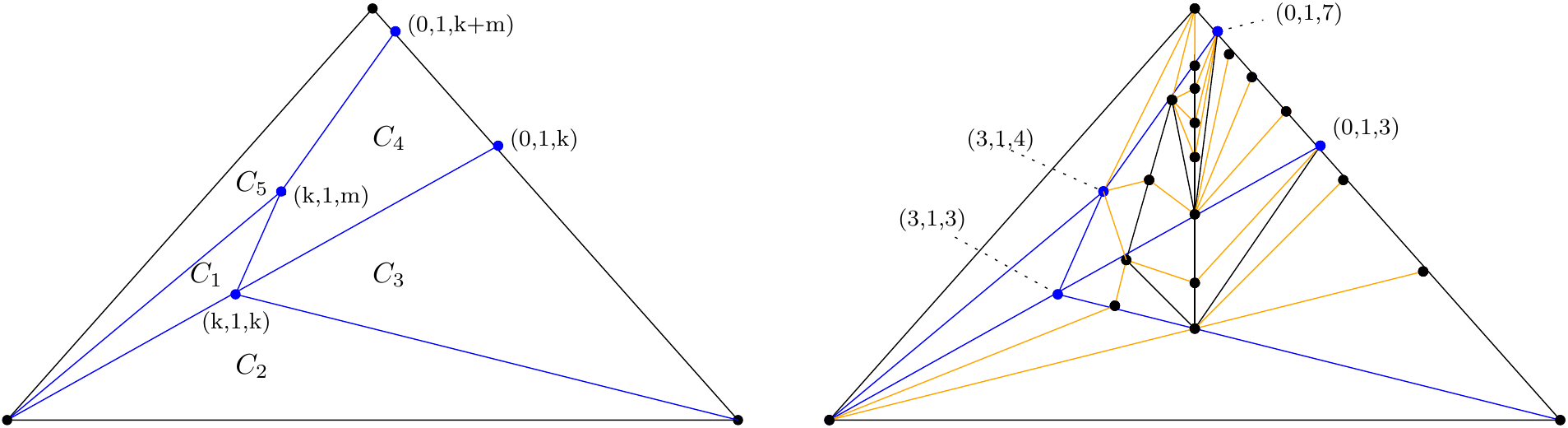}
\caption{Dual Newton fan of $A_{k-1,k-1,m-1}$ and an embedded resolution for $A_{2,2,3}$}
\end{figure}

\noindent Now let us decompose each subcone $C_i$ into regular cones: \\

\noindent {\bf Decomposition of $C_1$}: The cone $C_1$ contains the vectors $(k,1,\beta)$ for  $k\le \beta \le m-1$. They are on the skeleton of the fan. For $k\le \beta\le m-1$, we have:
$\left \vert \begin{array}{ccc}
\  k&k&1\\
\ 1&1&0\\
\ \beta&\beta+1&0\\
\end{array}\right \vert=1.$\\

\noindent {\bf Decomposition of $C_2$}: The cone $C_2$ contains the vectors $(1,1,1),\dots, (k,1,k)$ which are 
on the skeleton. For $0\le \alpha \le k-1$ we have:
$\left \vert \begin{array}{ccc}
\  1&\alpha&\alpha+1\\
\ 0&1&1\\
\ 0&\alpha&\alpha+1\\
\end{array}\right \vert=1.$\\

\noindent {\bf Decomposition of $C_3$}: To decompose the cone $C_3$, we first add successively an edge between the vectors $(k-1,1,k), (k-2,1,k-2),(k-3,1,k), \dots $ with the last vector being $(0,1,k)$ if $k$ is odd and with $(0,1,0)$ if $k$ is even. Then we obtain that the vectors $(\alpha,1,\alpha),\dots ,( \alpha,1,k)$ are in the same triangles (see Figure 7). Now let us add those vectors and the vectors on the associated edges successively.
 
 \begin{figure}[H]
\includegraphics [height=6cm, width =16cm]{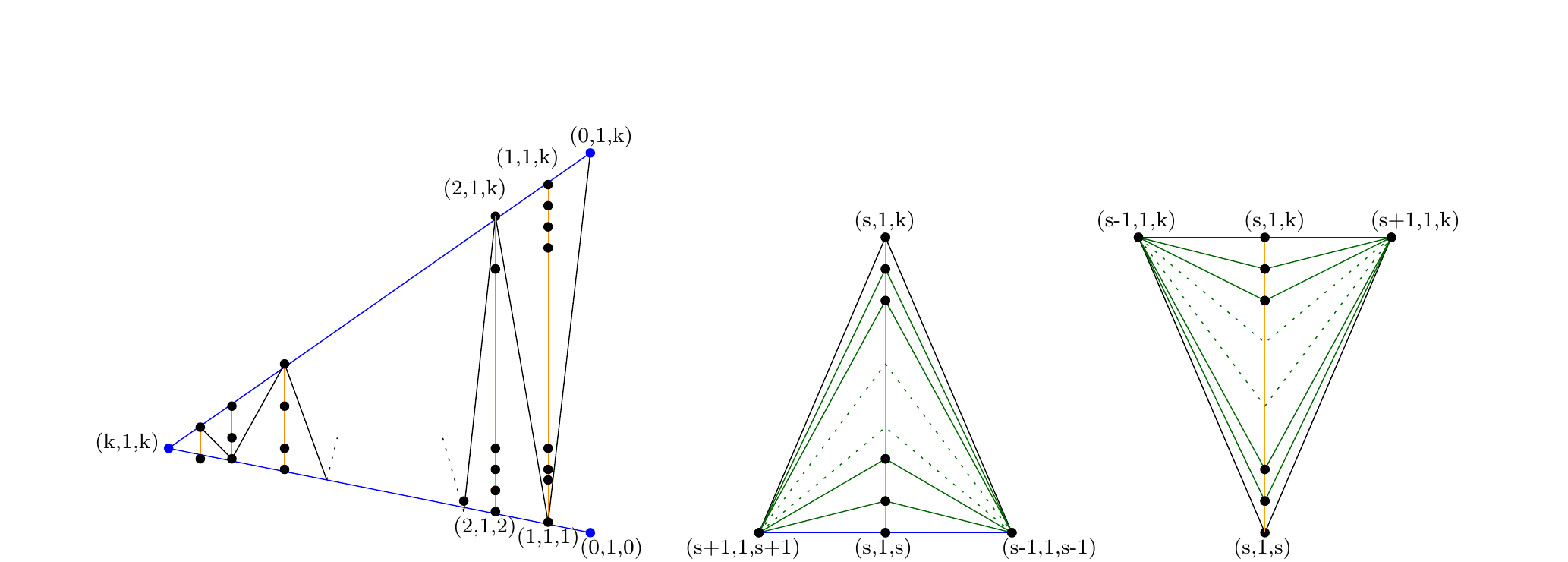}
\caption{Decomposition of the cone $C_3$ and of its two types of subcones}
\end{figure}

\noindent Each new subcone will  be regular as we only have one of the following two cases:
\begin{itemize}
\item {\bf Case 1}: for $\alpha \le \beta \le k-1$ we have
$$\left \vert \begin{array}{ccc}
\  \alpha-1&\alpha&\alpha\\
\ 1&1&1\\
\ k&\beta&\beta+1\\
\end{array}\right \vert=1\ \ and \ \ \left \vert \begin{array}{ccc}
\  \alpha+1&\alpha&\alpha\\
\ 1&1&1\\
\ k&\beta&\beta+1\\
\end{array}\right \vert=1$$
\item{\bf Case 2}:
$$\left \vert \begin{array}{ccc}
\  \alpha+1&\alpha&\alpha\\
\ 1&1&1\\
\ \alpha+1&\beta&\beta+1\\
\end{array}\right \vert=1\ \ and \ \ \left \vert \begin{array}{ccc}
\  \alpha-1&\alpha&\alpha\\
\ 1&1&1\\
\ \alpha-1&\beta&\beta+1\\
\end{array}\right \vert=-1$$
\end{itemize}

\noindent  {\bf Decomposition of $C_4$}: The cone $C_4$ is decomposed by adding successively the edges between the vectors $(k,1,m),(k-1,1,k),(k-2,1,m+2),\dots$ with the last vector being $(1,1,k)$ if $k$ is odd and with  $(1,1,k+m-1)$ if $k$ is even. Then let us add successively the vectors and the associated  edges  $(s,1,\alpha)$ for $k\le \alpha \le k+m-s$.
 \begin{figure}[H]
\includegraphics [height=5cm,width =15cm]{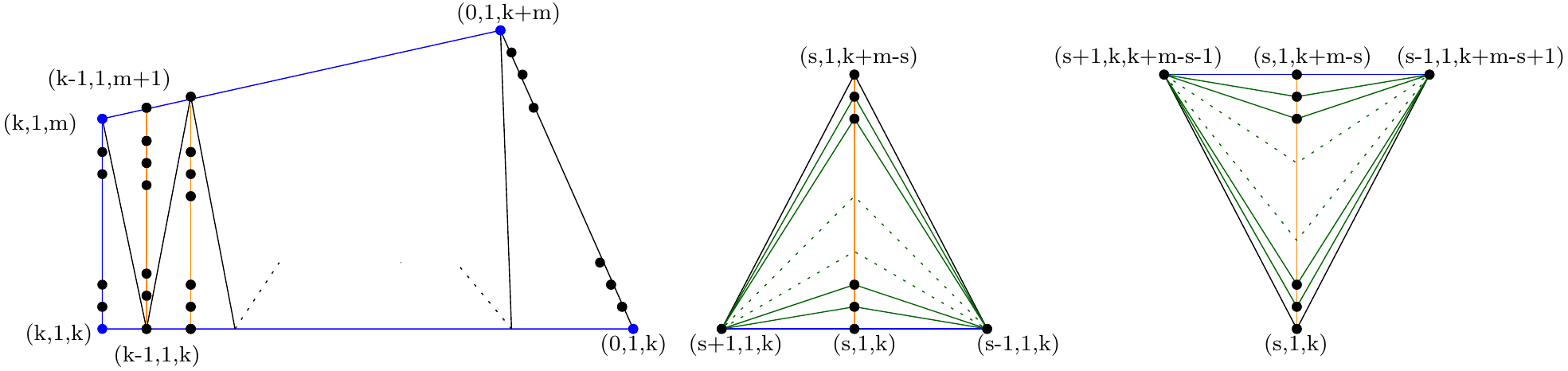}
\caption{Decomposition of the cone $C_4$ and of its two types of subcones}
\end{figure}

\noindent Each new subcone will  be regular as we have: for $0\le s \le k-1$ and for $k\le \beta \le k+m-s$,  
$$\left \vert \begin{array}{ccc}
\  s-1&s&s\\
\ 1&1&1\\
\ k&\beta&\beta+1\\
\end{array}\right \vert=1 \ \ and \ \ \left \vert \begin{array}{ccc}
\  s+1&s&s\\
\ 1&1&1\\
\ k&\beta&\beta+1\\
\end{array}\right \vert=1 $$ 
or
$$\left \vert \begin{array}{ccc}
\  s+1&s&s\\
\ 1&1&1\\
\ k+m-s-1&\beta&\beta+1\\
\end{array}\right \vert=1 \ \ and \ \ \left \vert \begin{array}{ccc}
\  s-1&s&s\\
\ 1&1&1\\
\ k+m-s+1&\beta&\beta+1\\
\end{array}\right \vert=1$$
\noindent {\bf Decomposition of $C_5$}: The cone $C_5$ contains the vectors $(s,1,k+m-s)$ for $0\le s \le k$ which are on the skeleton. For $0\le \alpha \le k$,   we have:  
$$\left \vert \begin{array}{ccc}
\  0&\alpha&\alpha+1\\
\ 0&1&1\\
\ 1&k+m-\alpha&k+m-1-\alpha\\
\end{array}\right \vert=1,  \left \vert \begin{array}{ccc}
\  k&1&0\\
\ 1&0&0\\
\ m&0&1\\
\end{array}\right \vert=1.$$
\end{proof}

\subsection{Jet Schemes and toric Embedded resolution of $A_{k-1,l-1,m-1}$ when $k\ge l \ge m$}

\noindent  The graph representing the irreducible components of the  jet schemes of $A_{k-1,l-1,m-1}$ projecting on the singular locus $\{y=z=0\}$ is given by Figure $9$ below. 
 \begin{figure}[H]
\setlength{\unitlength}{0.5cm}
\includegraphics[scale=0.7,height=14cm,width=16cm]{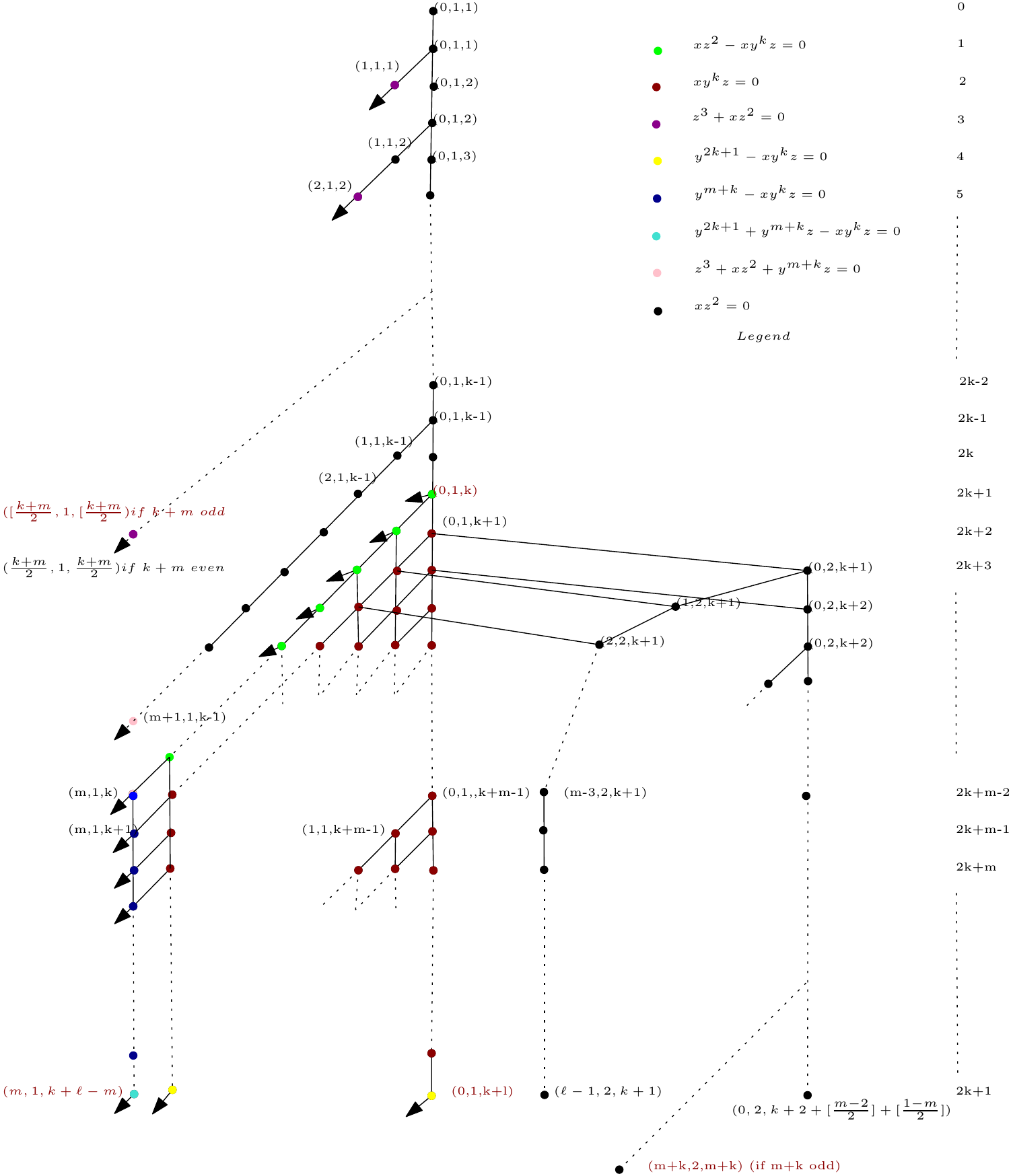}
\caption{ Jets schemes of $A_{k-1,l-1,m-1}$ }
\end{figure}

\begin{thm} \label{aa}
Let $X \subset \mathbb{C}^3$ be a surface of type $A_{k-1,l-1,m-1}$ with $k\geq l \geq m$. The monomial valuations associated with the vectors:  
\begin{itemize}
\item $(0,1,1), (0,1,2),\dots (0,1,k+l)$
\item $(s,1,s),\dots ,(s,1,l+k-s)$ $ 1\le s \le m-1$
\item $(m,1,m),\dots ,(m,1,k+l-m)$
\item $(m+r,1,m+r),\dots,(m+r,1,k-r)$ with $1\le r \le E(\frac{k-m}{2})$
\end{itemize}
belong to $EV(X)$. Moreover, these vectors give a toric birational map $\mu_\Sigma:Z_\Sigma\longrightarrow \mathbb{C}^3$ which is an embedded resolution of $X\subset \mathbb{C}^3$ (in the neighborhood of the origin) such that the components of the exceptional divisor of $\mu_\Sigma$ correspond to the monomial valuations defined by the vectors; hence they correspond to  irreducible components of the $m$-th jet schemes of $X$ (centered at the singular locus and the intersection of $X$ with the coordinate hyperplanes).    
\end{thm}

\begin{proof} As above, we first study the positions of the vectors given in theorem \ref{aa}:
\begin{itemize}
\item $(m+r,1,k-r)\in [(m+k,2,m+k),(m,1,k)]$: $$\left \vert \begin{array}{ccc}
\  m+k&m+k&m\\
\ 1&2&1\\
\ k+r&m+k&k\\
\end{array}\right \vert=\left \vert \begin{array}{ccc}
\  r&k&m\\
\ 0&1&1\\
\ -r&m&k\\
\end{array}\right \vert=\left \vert \begin{array}{ccc}
\  r&k&m\\
\ 0&1&1\\
\ 0&m+k&k+m\\
\end{array}\right \vert= 0$$
\item $(\alpha,1,k+l-\alpha)\in [(m,1,k+l-m),(0,1,k+l)]$ for $0\le \alpha \le m$ :
 $$\left \vert \begin{array}{ccc}
\  m&0&\alpha\\
\ 1&1&1\\
\ k+l-m&k+l&k+l-\alpha \\
\end{array}\right \vert=\left \vert \begin{array}{ccc}
\  m&0&\alpha\\
\ 0&1&0\\
\ -m&k+l&-\alpha\\
\end{array}\right \vert=0$$
\end{itemize}

 \noindent If $\frac{m+k}{2} \in \mathbb Z$, then the dual fan can be decomposed in the same way as for the case $A_{k-1,k-1,m-1}$.  Otherwise, we have to show the subcones containing the vector $(m+k,2,m+k)$ are regular. In this case $E(\frac{k-m}{2})=\frac{k-m-1}{2}$ and $(m+E(\frac{k-m}{2}),1,m+E(\frac{k-m}{2}))=(\frac{k+m-1}{2},1,\frac{k+m-1}{2})$. We have :
 $$\left \vert \begin{array}{ccc}
\ \frac{k+m-1}{2}&\frac{k+m-1}{2}&m+k\\
\ 1&1&2\\
\ \frac{k+m+1}{2}&\frac{k+m-1}{2}&m+k\\
\end{array}\right \vert=\left \vert \begin{array}{ccc}
\  0&\frac{k+m-1}{2}&m+k\\
\ 0&1&2\\
\ 1&\frac{k+m-1}{2}&m+k\\
\end{array}\right \vert=1\ \ and \ \ \left \vert \begin{array}{ccc}
\  1&\frac{k+m-1}{2}&m+k\\
\ 0&1&2\\
\ 0&\frac{k+m-1}{2}&m+k\\
\end{array}\right \vert=1.$$
\end{proof}

\begin{figure}[H]
\setlength{\unitlength}{0.6cm}
\includegraphics[scale=1,height=5cm,width=16cm]{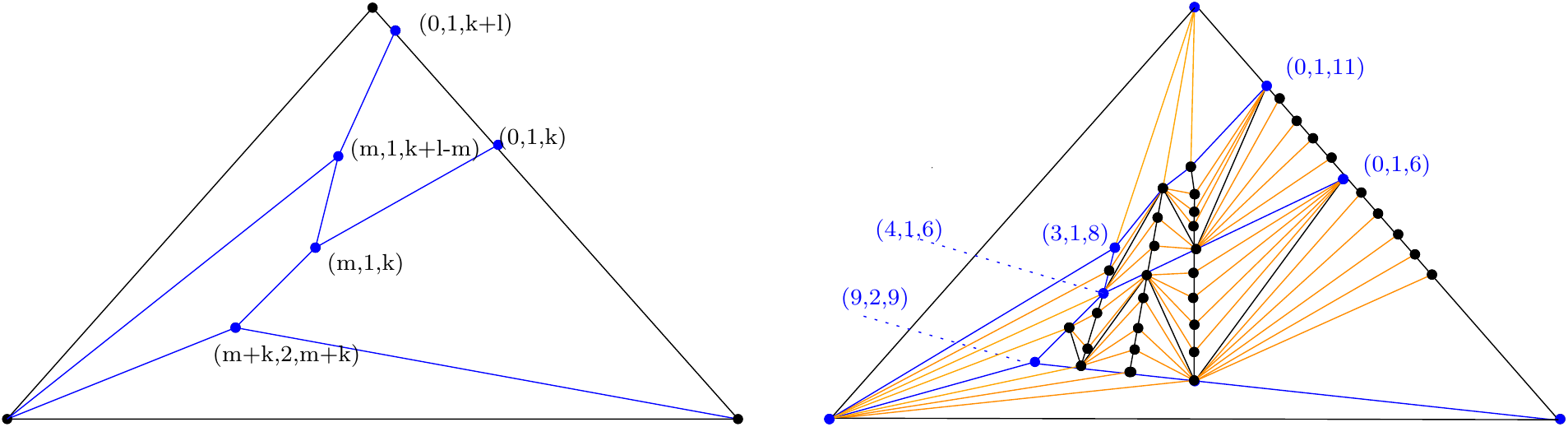}
\caption{ Dual Newton fan of $A_{k-1,l-1,m-1}$ with $k\geq l \geq m$, and a resolution of $A_{5,4,2}$ }
\end{figure}

\section{Jet Schemes and Toric Embedded Resolution of $B_{k-1,m}$}

\noindent  The singularity of $X\subset \mathbb{C}^3$ defined by  the equations:
  \begin{itemize}
	\item $m=2\ell$, 
	$$z^3+xz^2-(y^{k+1}+y^{\ell})y^kz-xy^{2k+1}=0,$$
	\item $m=2\ell-1$, 
	$$z^3+(x-y^{\ell-1})z^2-y^{2k+1}z-xy^{2k+1}=0.$$
	\end{itemize} is called $B_{k-1,m}$-type singularity with  $k\geq 2$ and $m\geq 3$.\\
	 {\bf In the case where $m=2l$},  the jet schemes and the toric embedded resolution behaves as in the case of $A_{m,k,l}$; so, let's just present the jet graph presenting the irreducible components of the jets schemes projecting over the singular locus $\{y=z=0\}$ and the axis $\{x=z=0\}$ included in $X$:
\begin{figure}[H]
\setlength{\unitlength}{0.28cm}
\includegraphics[scale=0.7,height=14cm,width=14cm]{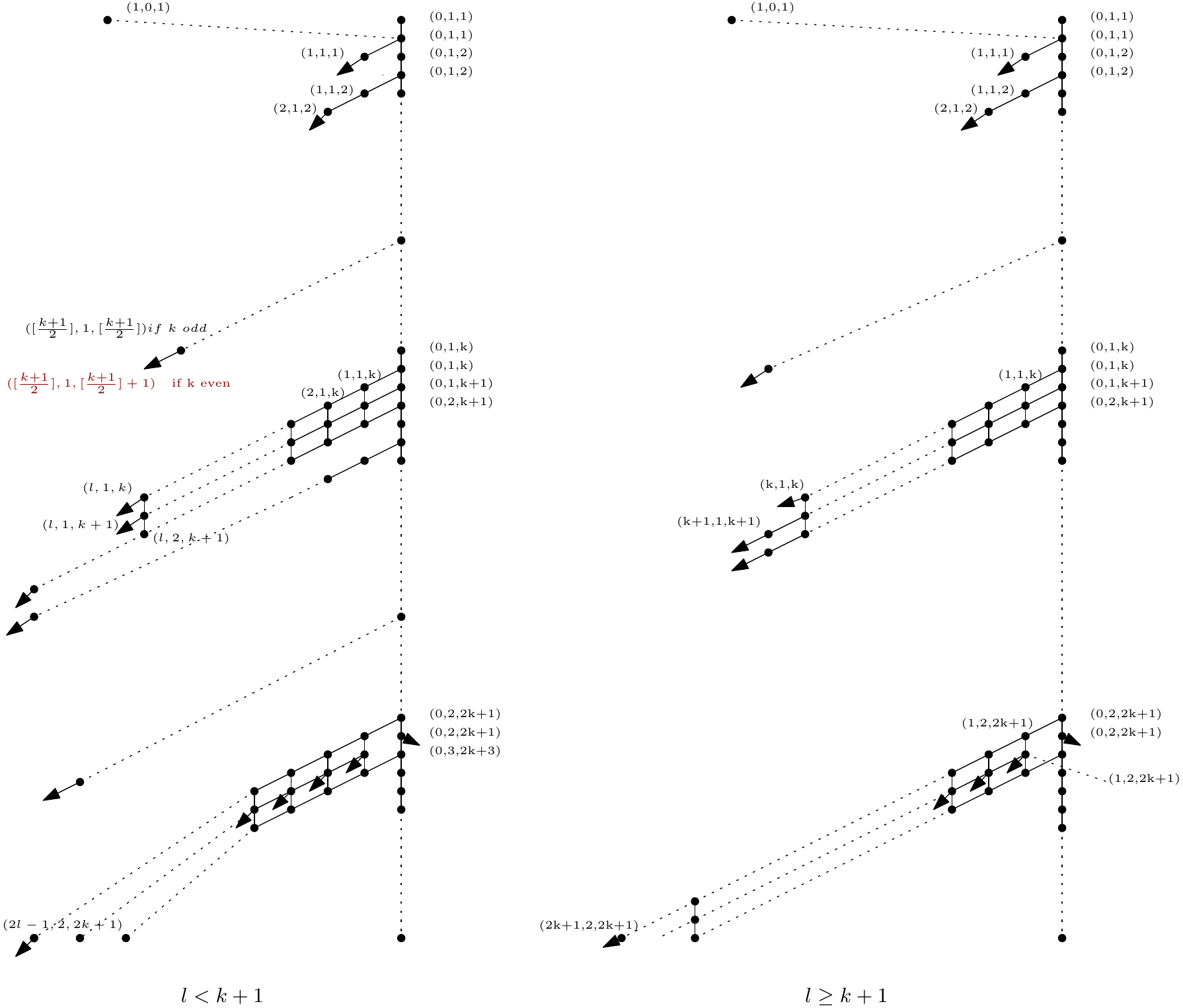}
\caption{Jets schemes of $B_{k-1,m}$ when $m=2l$ }
\end{figure}

\begin{thm}
Let $X \subset \mathbb{C}^3$ be a surface of type $B_{k-1,2l}$. The monomial valuations associated with the vectors:  
\begin{itemize}
\item $(0,1,1), (0,1,2),\dots, (0,1,k+1)$
\item $(1,1,1),\dots,(1,1,k+1$
\item \dots
\item $(l,1,l),\dots,(l,1,k+1)$
\item $(l+1,1,l+1),\dots,(l+1,1,k-1)$
\item $(l+2,1,l+2),\dots,(l+1,1,k-2)$
\item \dots
\item $(E((l+k)/2),1,E((l+k)/2))$, and $(E((l+k)/2),1,E((l+k)/2)+1)$ if $k+l$ is odd.
\item $(0,2,2k+1)\dots(2l-l,2,2k+1)$
\end{itemize}
belong to $EV(X).$ Moreover, these vectors give a toric birational map $\mu_\Sigma:Z_\Sigma\longrightarrow \mathbb{C}^3$ which is an embedded resolution of $X\subset \mathbb{C}^3$ (in the neighborhood of the origin) such that the components of the exceptional divisor of $\mu_\Sigma$ correspond to the monomial valuations defined by them; hence they correspond to the irreducible components of the $m$-th jet schemes of $X$ (centered at the singular locus and the intersection of $X$ with the coordinate hyperplanes).    
\end{thm}

\noindent The vectors given in the theorem allows us to decompose the corresponding dual Newton fan into regular subcones and find an embedded resolution of the singularity.

\begin{figure}[H]
\setlength{\unitlength}{0.5cm}
\scalebox{0.65}{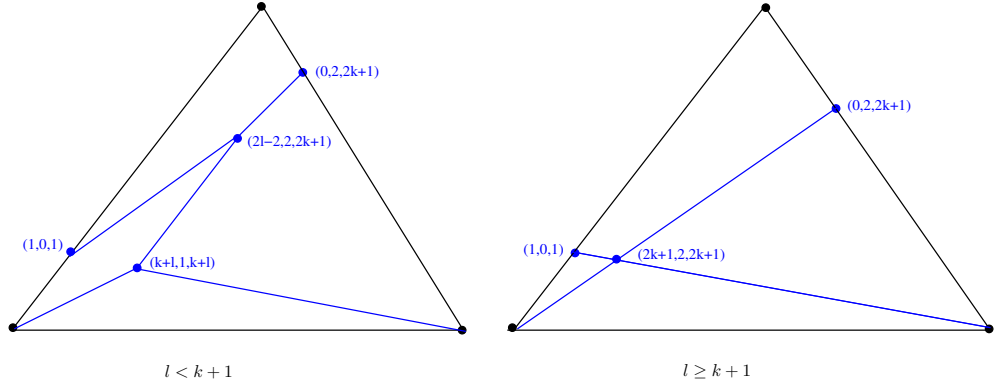}
\caption{ Dual Newton fans of  $B_{k-1,m}$ when $m=2l$}
\end{figure}

\noindent Two embedded resolutions for two special cases look as the following:

\begin{figure}[!ht]
\includegraphics [width =14cm]{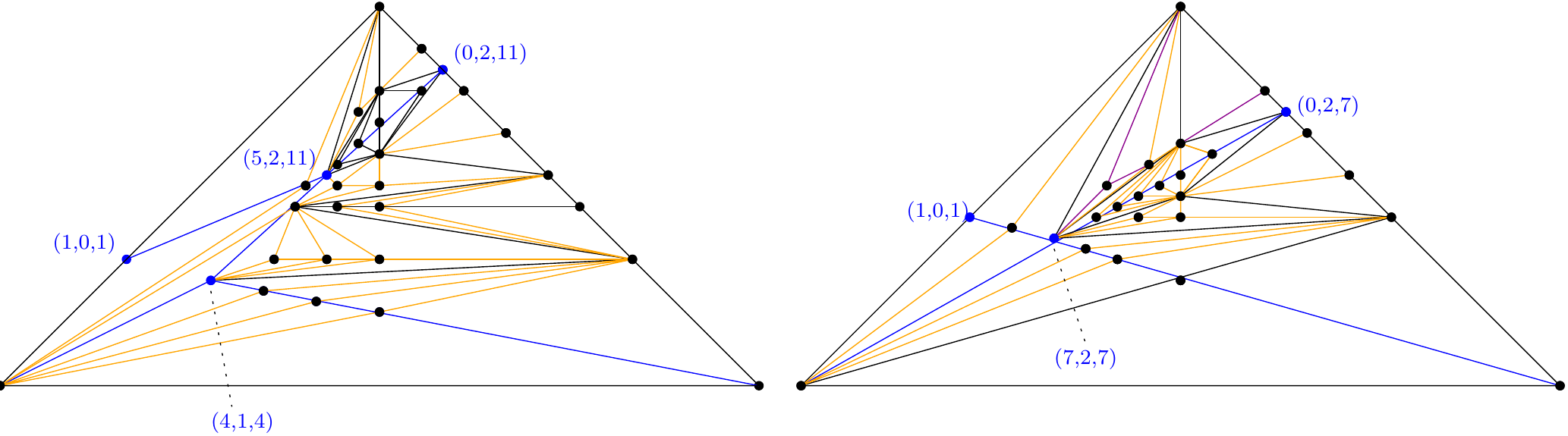}
\caption{Embedded resolution of $B_{4,6}$ and  of $B_{2,10}$}
\end{figure}

\noindent {\bf In the case of $B_{k-1,2l-1}$}, there is an amazing subclass (see Section $2$ below) for which the jet schemes give a resolution which is not a subdivision of the dual Newton fan of the singularity. So this case needs to be treated in details. There are two sub-cases to be considered which are the cases $k+1\le l$ and $k\geq l$. \\
 {\bf Let us first treat the case $k+1\le l$}: we start by computing the irreducible components of the jet schemes projecting on the singular locus $\{y=z=0\}$ and the axe $\{x=z=0\}$ included in $X$. And, by computing the associated vectors we obtain the figure \ref{jetB}:

\begin{thm}
Let $X$ be of type  $B_{k-1,2l-1}$ with $k+1\le l$. The monomial valuations associated with the vectors:
\begin{itemize}
\item $(0,1,1), (0,1,2),\dots, (0,1,k+1)$
\item $(1,1,1),\dots,(1,1,k+1)$
\item \dots
\item $(k,1,k),(k,1,k+1)$
\item $(k+1,1,k+1)$
\item $(0,2,2k+1)\dots(2k+1,2,2k+1)$
\end{itemize}
belong to $EV(X).$ Moreover there exists a toric birational map $\mu_\Sigma:Z_\Sigma\longrightarrow \mathbb{C}^3$ which is an embedded resolution of $X\subset \mathbb{C}^3$ such that the irreducible components of the exceptional divisor of $\mu_\Sigma$ correspond to the irreducible components of the $m$-th jet schemes of $X$ (centered at the singular locus and the intersection of $X$ with the coordinate hyperplane). This yields a construction of $\mu_\Sigma $ (not canonical).     
\end{thm}
\begin{figure}[H]
\includegraphics[scale=0.7,height=12cm,width=15cm]{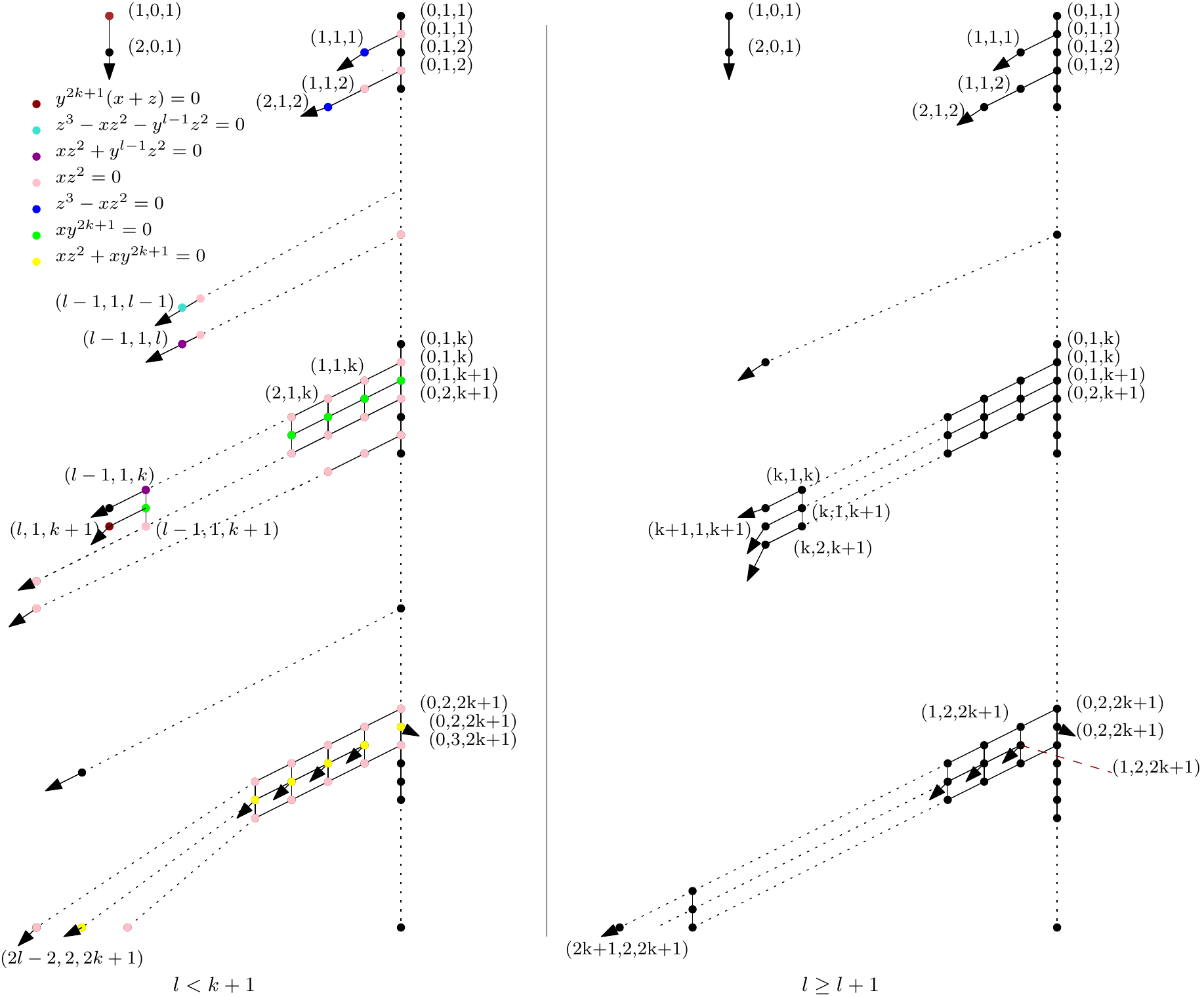}
\caption{Jets schemes of $B_{k-1,2l-1}$ }
\label{jetB}
\end{figure}

\noindent The computations are similar to the case $B_{k-1,2l}$; the associated vectors with the jet schemes give a subdivision of the dual Newton fan, thus an embedded resolution, of the singularity. 

\begin{thm} Let $X$ be of type  $B_{k-1,2l-1}$ for $l\le k$. The monomial valuations associated with the vectors  
\begin{itemize}
\item $(0,1,1), (0,1,2),\dots, (0,1,k+1)$
\item $(1,1,1),\dots,(1,1,k+1)$
\item \dots
\item $(l-1,1,l-1),\dots, (l-1,1,k+1)$
\item $(l,1,k+1)$
\item $(0,2,2k+1)\dots(2l-2,2,2k+1)$
\end{itemize}
belong to $EV(X).$ Moreover there exists a toric birational map $\mu_\Sigma:Z_\Sigma\longrightarrow \mathbb{C}^3$ which is an embedded resolution of $X\subset \mathbb{C}^3$ such that the irreducible components of the exceptional divisor of $\mu_\Sigma$ correspond to the irreducible components of the $m$-th jet schemes of $X$ (centered at the singular locus and the intersection of $X$ with the coordinate hyperplane). This yields a construction  of $\mu_\Sigma $ (not canonical).     
\end{thm}

\begin{figure}[H]
\setlength{\unitlength}{0.40cm}
\includegraphics[scale=0.5,height=6cm,width=13cm]{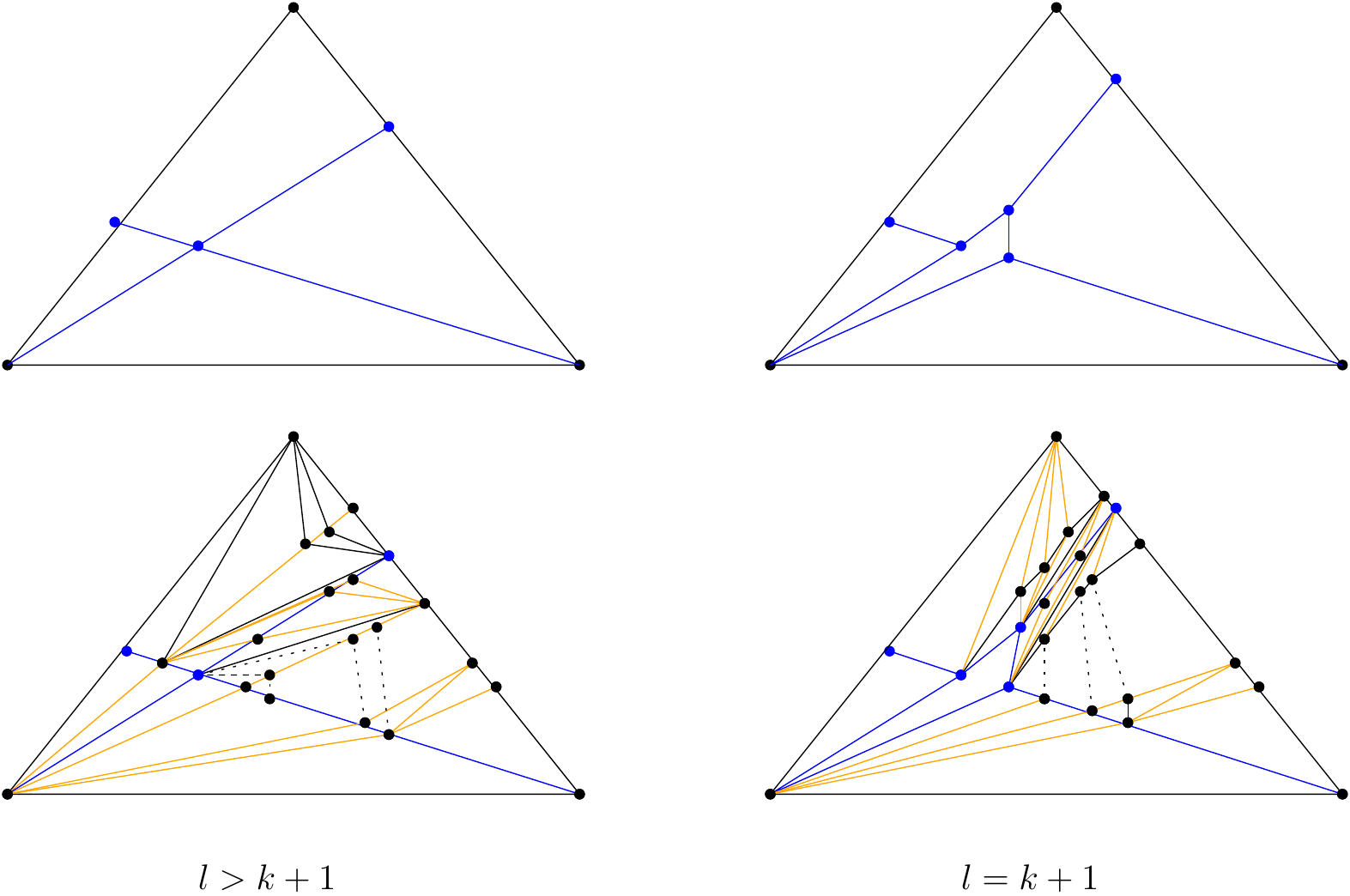}
\caption{Dual Newton fan of $B_{k-1,2l-1}$ for $l>k+1$ ($resp.\ l=k+1$) and an embedded resolution}
\end{figure}
\noindent {\bf In the case where $X$ is of type $B_{k-1,2l-1}$ with $l\le k$}, the corresponding dual Newton fan of the singularity is  given with the right hand figure of Figure $16$. 
\begin{figure}[H]
\setlength{\unitlength}{0.40cm}
\includegraphics[scale=0.75,height=4cm,width=17cm]{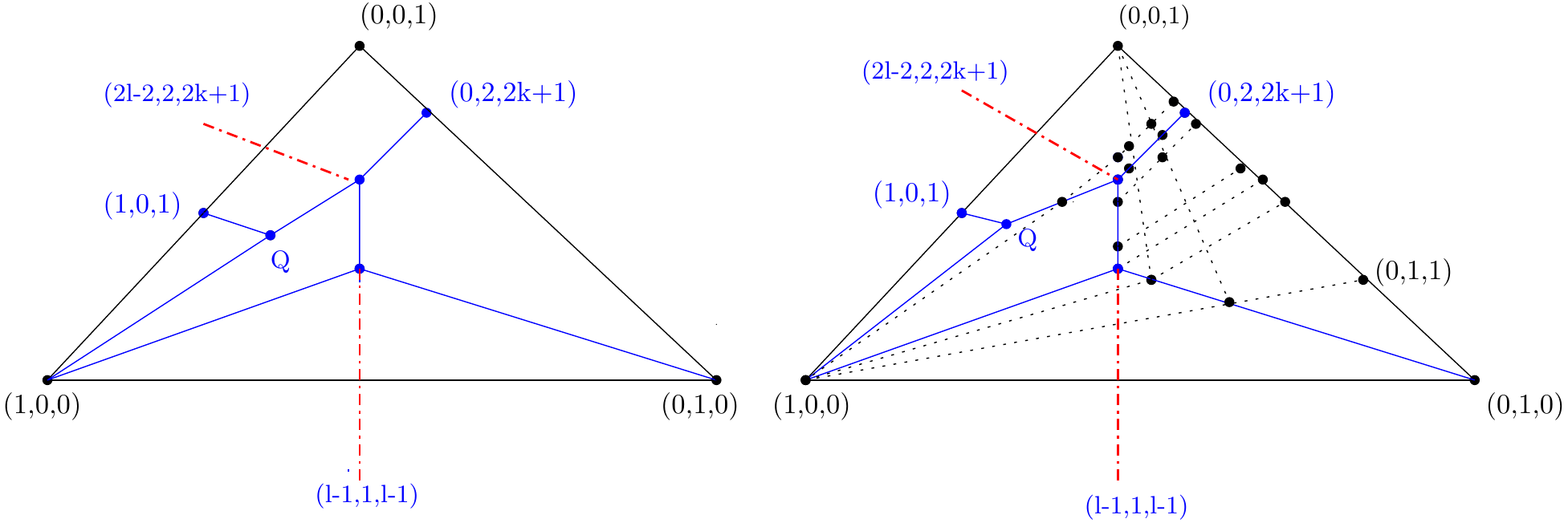}
\caption{Dual Newton fan of $B_{k-1,2l-1}$ with $l\le k$ and it is with the vectors of Theorem $6.3$}
\end{figure}

\begin{rem}
The set of vectors above does not contain the vector $Q=(2k-l+2,1,2k-l+2)$, thus the decomposition obtained by these vectors will not be  a regular decomposition of the dual Newton fan of the singularity.
\end{rem}

\begin{proof} Consider the polygons $J=[(1,0,0),(0,1,0),(0,0,1),$ $(1,0,1),(l,1,k+1),(2l-2,2,2k+1),(l-1,1,l-1)]$ and $K=[(1,0,0),(1,0,1),(l,1,k+1),(2l-2,2,2k+1),(l-1,1,l-1)]$ in the dual Newton fan of the singularity. In $J$, the vectors obtained from the jet schemes give a regular subdivision of this polygon (following the computations of $B_{k-1,2l}$). As $J$ is a sub-polygon of the fan, the strict transform of $X$ is regular on these charts. In $K$, 
we find a subdivision by adding an edge from $(1,0,0)$ to $(l,1,k+1)$, another edge from $(1,0,0)$ to $(l-1,1,s)$ for $l-1\le s\le k$ and another edge from $(l,1,k+1)$ to $(l-1,1,k)$. In this way, we obtain a regular subdivision of $K$.

\noindent  {\bf Since $K$ is not compatible with the dual Newton fan, we cannot use Varchenko's theorem to deduce the smoothness of the strict transform of $X$  in the charts corresponding to the subdivision of $K$ by the toric map. So, we should prove this fact:}

\begin{itemize}
\item For this, let us first consider the cone $[(1,0,0),(l-1,1,s),(l-1,1,s+1)]$ for $l-1\le s<k$; the monoidal transformation corresponding to it is:
$$\left \{ \begin{array}{c}
\  x=x_1y_1^{l-1}z_1^{l-1}\\
\  y=y_1z_1\\
 \ z=y_1^{s}z_1^{s+1}\\
\end{array}\right .$$

\noindent Then the total transform of $B_{k-1,2l-1}$ is defined by:
$$\{y_1^{2s+l-1}z_1^{2s+l+1}(y_1^{s-l+1}z_1^{s-l+2}-x_1-y_1^{2k-s-l+2}z_1^{2k-s-l+1}-x_1y_1^{2k-2s+1}z_1^{2k+2s-1})=0\}$$ The strict transform is smooth and transversal to the exceptional divisors defined by $y_1=0$ and $z_1=0.$
\item Now let us consider the cone $[(1,0,0),(l-1,1,k),(l,1,k+1)]$; the monoidal transformation corresponding to it is:
$$\left \{ \begin{array}{c}
\  x=x_1y_1^{l-1}z_1^{l}\\
\  y=y_1z_1\\
 \ z=y_1^{k}z_1^{k+1}\\
\end{array}\right .$$
Then the total transform of $B_{k-1,2l-1}$ is :
$$\{y_1^{2k+l-1}z_1^{2k+l+1}(y_1^{k-l+1}z_1^{k-l}-x_1z_1-1-y_1^{k-l+2}z_1^{k-l+1}-x_1y_1)=0\}.$$ 
The strict transform is smooth and transversal to the exceptional divisors defined by $y_1=0$ and $z_1=0.$
\item Finally let us consider the cone $[(1,0,0),(1,0,1),(l,1,k+1)]$; the monoidal transformation corresponding to it is:
$$\left \{ \begin{array}{c}
\  x=x_1y_1z_1^{l}\\
\  y=z_1\\
 \ z=y_1z_1^{k+1}\\
\end{array}\right .$$
Then the total  transform of $B_{k-1,2l-1}$ is :
$$\{y_1z_1^{2k+l+1}(y_1^2z_1^{k+2-l}-x_1y_1^2z_1-y_1-z_1^{k-l+1}-x_1)=0\}.$$ The strict transform is smooth and transversal to the exceptional divisors defined by $y_1=0$ and $z_1=0.$
\end{itemize}
\begin{figure}[H]
\setlength{\unitlength}{0.7cm}
\includegraphics[scale=0.8,height=4cm,width=14cm]{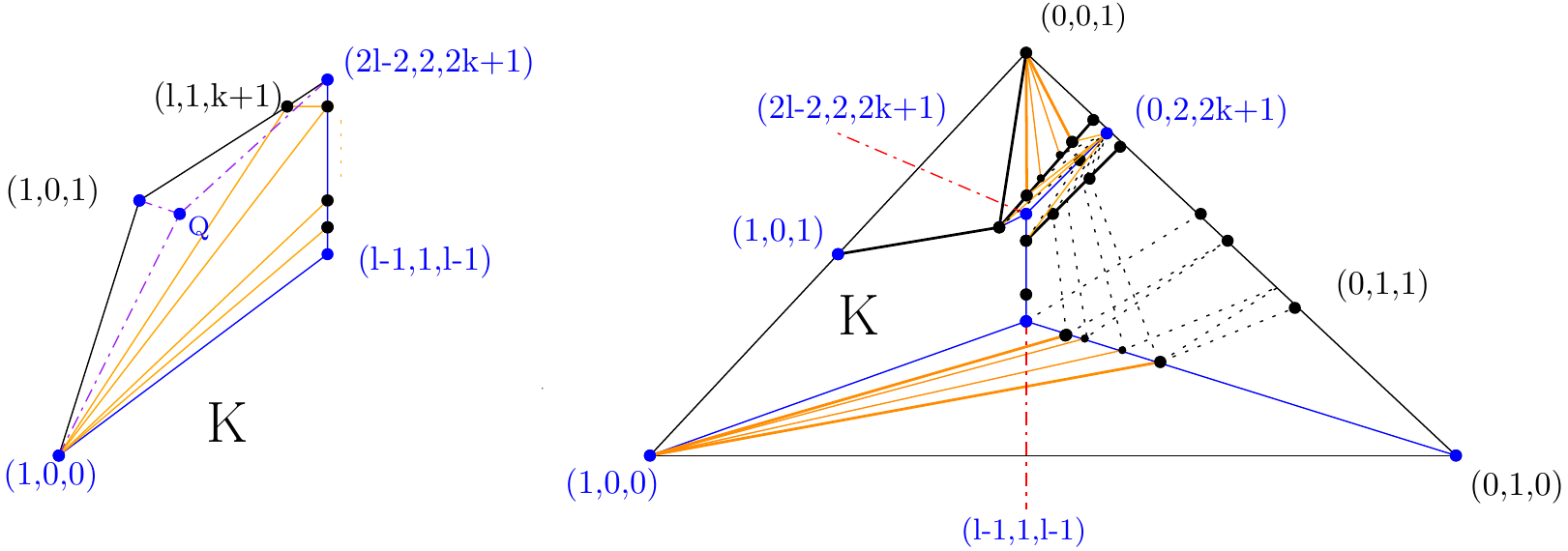}
\caption{The polygon $K$ and an embedded resolution of  $B_{k-1,2l-1}$ with $l\le k$}
\end{figure}
\end{proof}

\section{Jet Schemes and Toric Embedded Resolution of $C_{k-1,l+1}$}

\noindent  The singularity of $X\subset \mathbb{C}^3$ defined by  the equation:
$$z^3+xz^2-\ell x^{\ell-1}y^{2k}z-(x^\ell+y^2)y^{2k}=0$$
is called $C_{k-1,l+1}$-type singularity where $k\geq 1$ and $\ell\geq 2$. For $k=3q-1$, we obtain the jet graph given in Figure \ref{jetC} which represents the irreducible components of the jet schemes of $C_{k-1,l+1}$ projecting on the singular locus $\{y=z=0\}$.
 \begin{figure}[H]
\includegraphics[height=15cm,width=15cm]{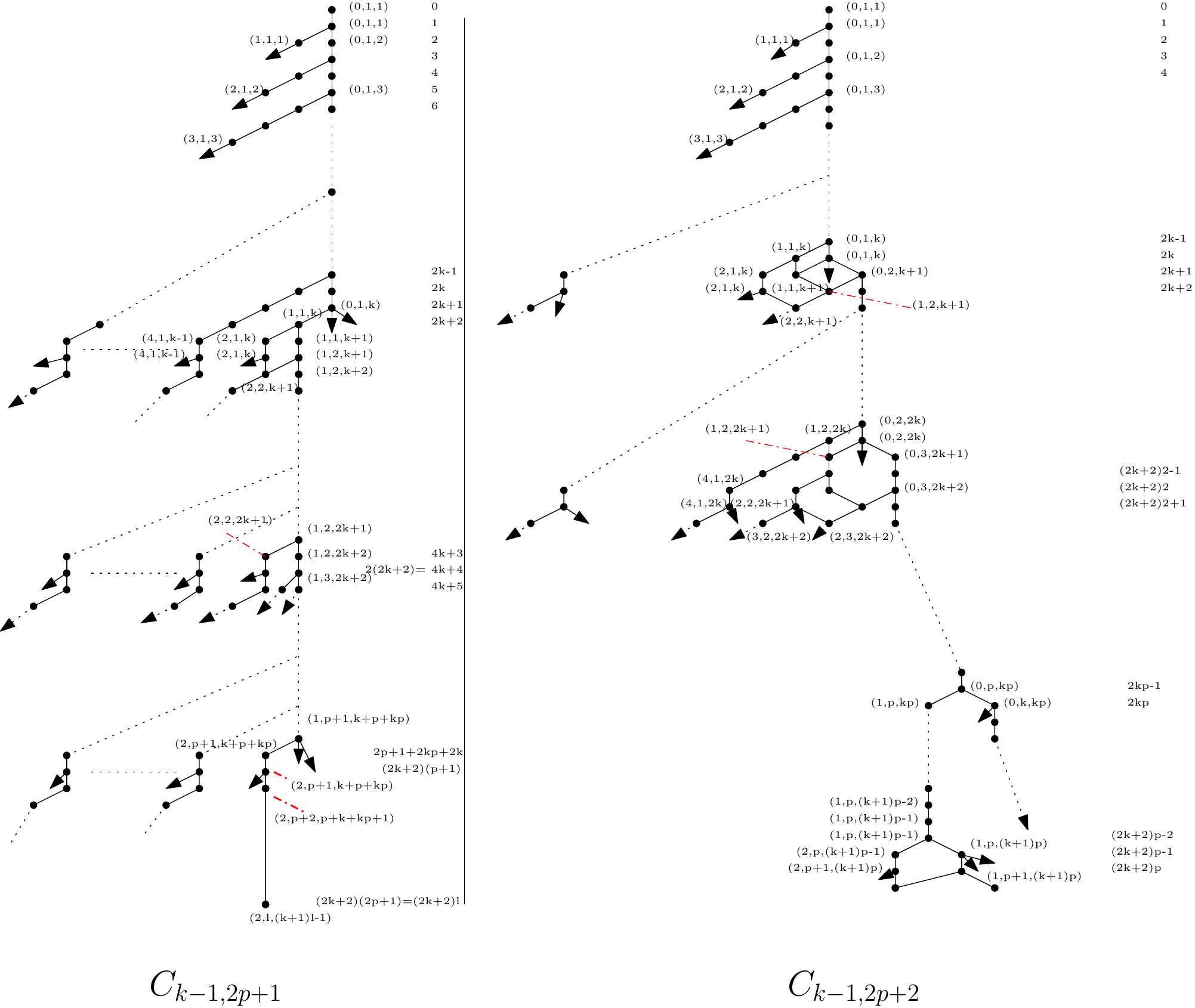}
\caption{Jet schemes of $C_{k-1,2p+2}$ with $k=3q-1$}
\label{jetC}
\end{figure}

\begin{thm}
Let $X$ be a surface singularity of type  $C_{k-1,l+1}$. The monomial valuations associated with the vectors:
\begin{itemize}
\item for { $k=3q-1$} and $l=2p$
\begin{itemize}
\item $(0,1,1), (0,1,2)...(0,1,k)$
\item $(1,1,1), ... ,(1,1,k),(1,1,k+1)$
\item $(2,1,2),... ,(2,1,k)$
\item $(3,1,3),... ,(3,1,k-1)$
\item $(4,1,4),... ,(4,1,k-1)$
\item \dots
\item $(2q-1,1,2q-1), (2q-1,1,2q)$
\item $(2q,1,2q)$
\item $(1,1,k),(1,2,2(k+1)-1),\dots,(1,p,(k+1)p-1)$
\item $(2,1,k),(2,2,2(k+1)-1),(2,3,3(k+1)-1),\dots,(2,l,(k+1)l-1)$
\item $(1,1,k+1),(1,2,2(k+1)),\dots,(1,p,(k+1)p)$
\end{itemize}
\item for { $k=3q-1$} and $l=2p+1$
\begin{itemize}
\item $(0,1,1), (0,1,2)...(0,1,k)$
\item $(1,1,1), ... ,(1,1,k),(1,1,k+1)$
\item $(2,1,2),... ,(2,1,k)$
\item $(3,1,3),... ,(3,1,k-1)$
\item $(4,1,4),... ,(4,1,k-1)$
\item \dots
\item $(2q-1,1,2q-1), (2q-1,1,2q)$
\item $(2q,1,2q)$
\item $(1,1,k),(1,2,2(k+1)-1),\dots,(1,p,(k+1)p-1),(1,p,(k+1)(p+1)-1)$
\item $(2,1,k),(2,2,2(k+1)-1),(2,3,3(k+1)-1),\dots,(2,l,(k+1)l-1)$
\item $(1,1,k+1),(1,2,2(k+1)),\dots,(1,p,(k+1)p)$
\end{itemize}
\end{itemize}
belong to $EV(X).$ Moreover there exists a toric birational map $\mu_\Sigma:Z_\Sigma\longrightarrow \mathbb{C}^3$ which is an embedded resolution of $X\subset \mathbb{C}^3$ such that the irreducible components of the exceptional divisor of $\mu_\Sigma$ correspond to the irreducible components of the $m$-th jet schemes of $X$ (centered at the singular locus and the intersection of $X$ with the coordinate hyperplane). This yields a construction (not canonical) of $\mu_\Sigma.$    
\end{thm}

The embedded resolutions are represented on the figure below.
\begin{figure}[H]
\setlength{\unitlength}{0.5cm}
\includegraphics[height=5cm,width=17cm]{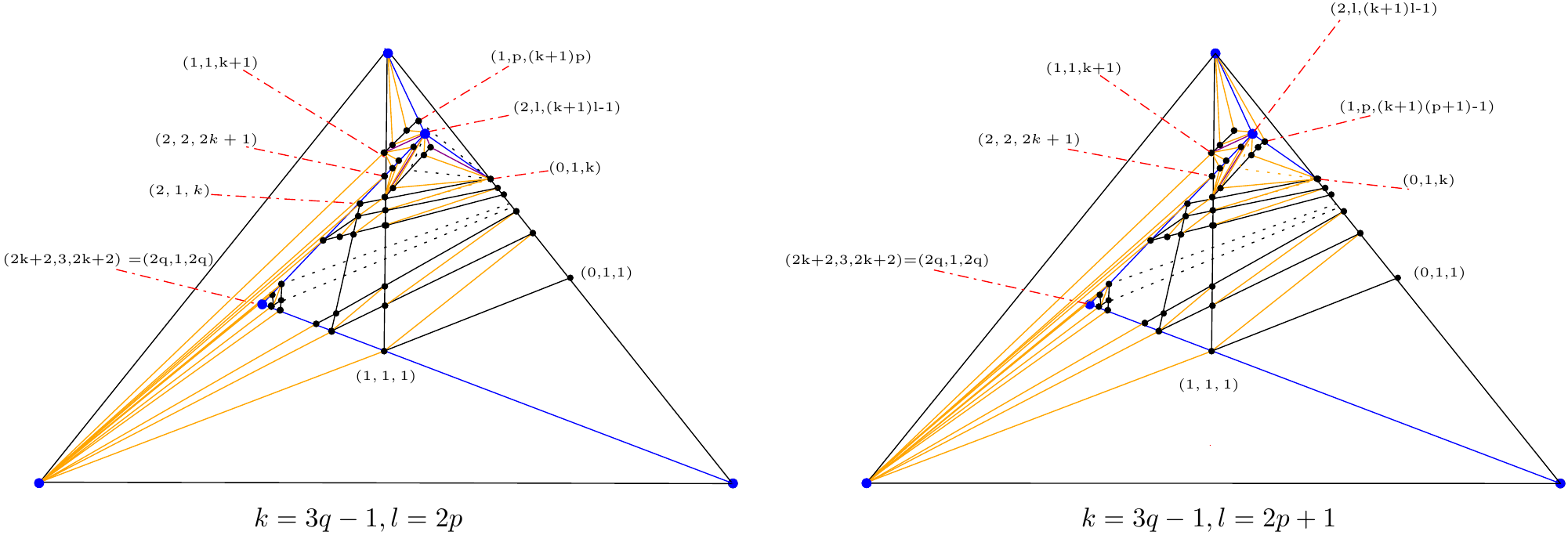}
\caption{An embedded resolution of $C_{k-1,l+1}$ when $k=3q-1$ and $l=2p$ or $l=2p+1$ }
\end{figure}

\section{Jet Schemes and Toric Embedded Resolution of $D_{k-1}$}

\noindent  The singularity of $X\subset \mathbb{C}^3$ defined by the equation:
$$z^3+(x+y^{2k})z^2+(2xy^{k}-y^2)y^kz+x^2y^{2k}=0$$
is called $D_{k-1}$-type singularity with $k\geq 1$. The jet graph is given in Figure $20$ where the irreducible components of the jet schemes of $D_{k-1}$ projecting on the singular locus $\{(y=z=0\}$ and the axe $\{x=z=0\}$ included in $X$:

\begin{figure}[H]
\setlength{\unitlength}{0.28cm}
\includegraphics[height=16cm,width=15cm]{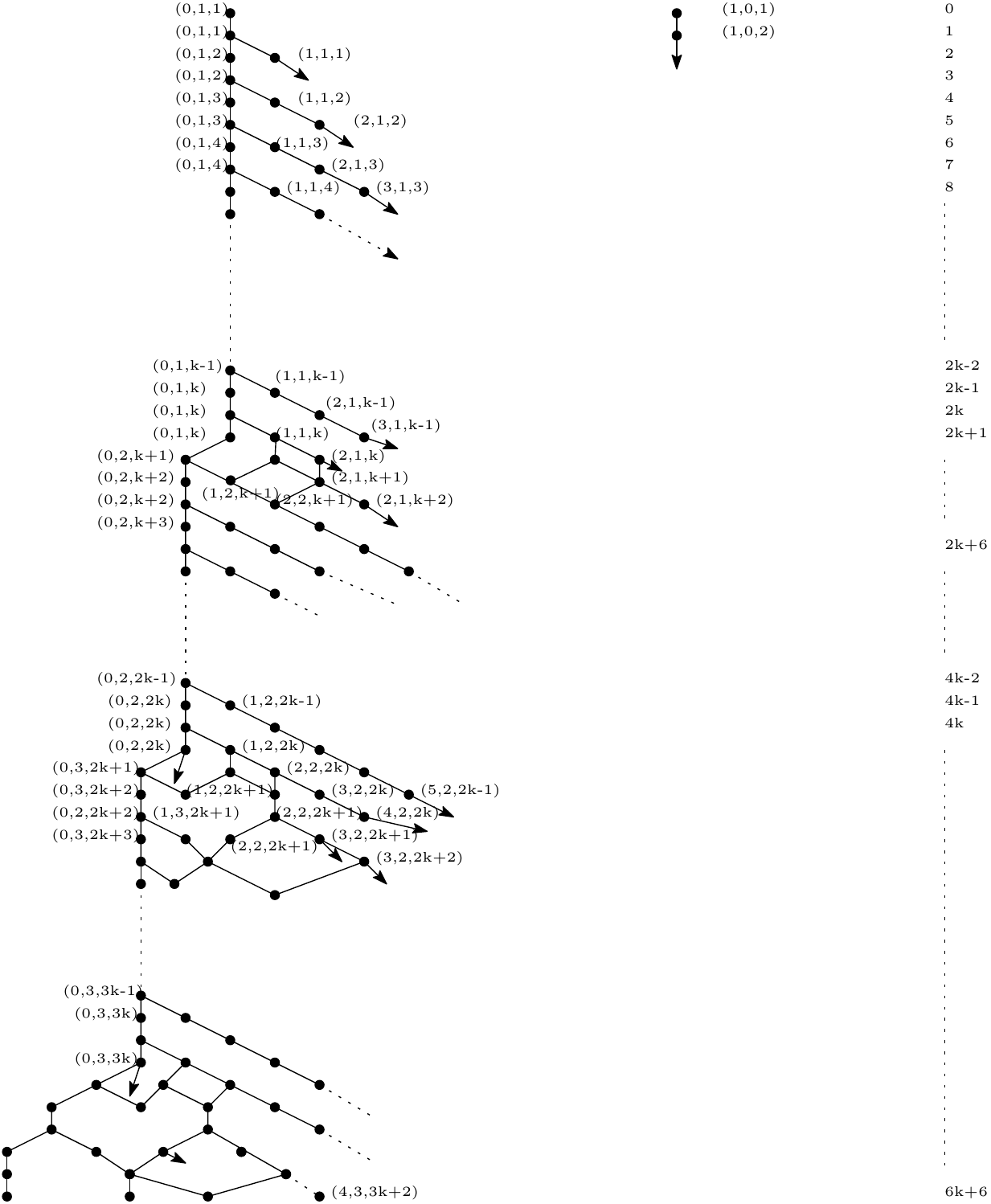}
\caption{Jet schemes of $D_{k-1}$ }
\end{figure}

\begin{thm}
Let $X$ be a surface singularity of type $D_{k-1}$. The monomial valuations associated with the following vectors
belong to $EV(X).$ Moreover there exists a toric birational map $\mu_\Sigma:Z_\Sigma\longrightarrow \mathbb{C}^3$ which is an embedded resolution of $X\subset \mathbb{C}^3$ such that the irreducible components of the exceptional divisor of $\mu_\Sigma$ correspond to the irreducible components of the $m$-th jet schemes of $X$ (centered at the singular locus and the intersection of $X$ with the coordinate hyperplane). This yields a construction (not canonical) of $\mu_\Sigma.$     
\begin{itemize}
\item { $(1,0,1), (1,0,2)$}
\item $(0,1,1), (0,1,2)...(0,1,k)$
\item $(1,1,1), ... ,(1,1,k),(2,2,2k+1)),(1,1,k+1)$
\item $(2,1,2),... ,(2,1,k+2)$
\item $(3,1,3),... ,(3,1,k-1)$
\item \dots
\item $(m,1,m), (m,1,m+1),(m,1,m+2)$
\item $(m+1,1,m+1)$
\item $(3,2,2k+1),(3,2,2k+2)$
\end{itemize}
\noindent When $k$ is odd, we should add two more vectors: $(m+1,1,m+2), (k+2,2,k+2)$
where  $m=E(\frac{k}{2})$).
\end{thm}
\noindent These vectors placed in the dual Newton fan give the regular subdivision:
\begin{figure}[H]
\setlength{\unitlength}{0.5cm}
\includegraphics[scale=0.8,height=6cm,width=14cm]{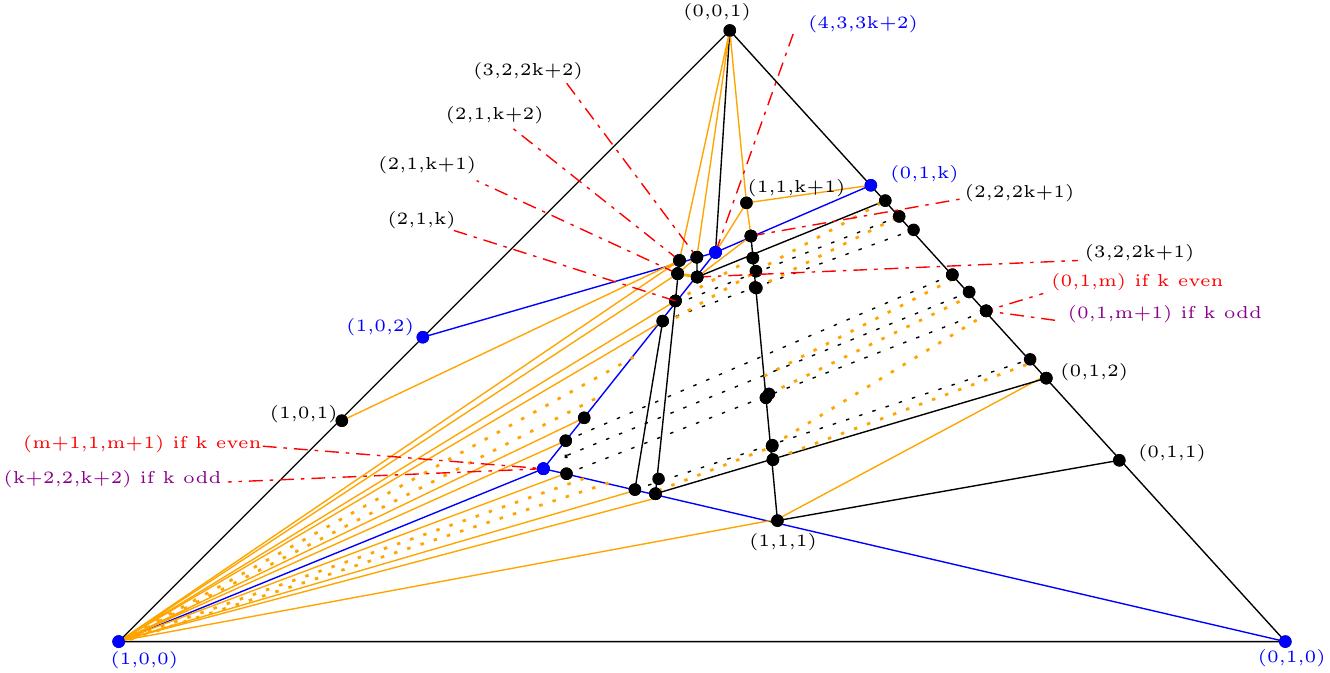}
\caption{Embedded resolutions of $D_{k-1}$ for $k=2m$ and $k=2m+1$ }
\end{figure}

\section{Jet Schemes and Toric Embedded Resolution of  $E_{7,0}$}

\noindent  The singularity of $X\subset \mathbb{C}^3$ defined by  the equation:
$$z^{3}+x^{2}yz+y^{4}=0$$ 
is called $E_{7,0}$-type singularity. The singular locus is $\{y=z=0\}$.

\begin{thm}
Let $X$ be a surface singularity of type $E_{7,0}$. The monomial valuations associated with the vectors: 
$\{(0,1,1),(0,2,1),(0,1,2), {(0,1,3),}$ $(1,1,1),$ $(1,1,2),$ $(1,2,2),$ $(1,2,3),$ $(1,2,4),$ $(2,2,3),$ $(2,3,4),$ $(2,3,5),$ $(3,3,4),$ $(3,4,5),$ $(3,4,6),$ $(4,5,7),$ $(5,6,8)\}$ belong to $EV(X).$ There exists a toric birational map $\mu_\Sigma:Z_\Sigma\longrightarrow \mathbb{C}^3$ which is an embedded resolution of $X\subset \mathbb{C}^3$ such that the irreducible components of the exceptional divisor of $\mu_\Sigma$ correspond to the irreducible components of the $m$-th jet schemes of $X$ (centered at the singular locus and the intersection of $X$ with the coordinate hyperplane). Moreover this yields a construction (not canonical) of $\mu_\Sigma.$    
\end{thm}

\noindent  Following almost the same process as in the case of $E_{6,0}$, we continue until $m=22$ to obtain the 
following jet graph: 

\begin{figure}[H]
	\setlength{\unitlength}{0.5cm}
\includegraphics[scale=0.8,height=12cm,width=12cm]{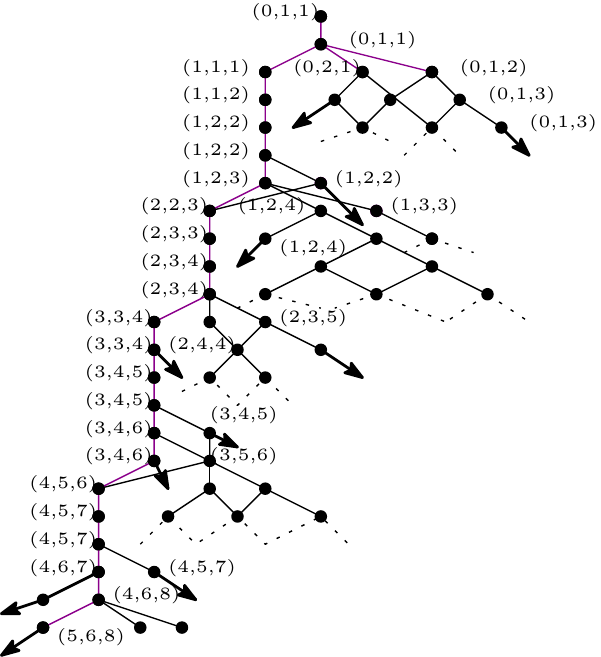}
	\caption{ Jet schemes of $E_{7,0}$}
\end{figure}

\noindent The vectors corresponding to the irreducible jet schemes give the following subdivision, which is an embedded resolution of $X$:

 \begin{figure}[H]
\setlength{\unitlength}{0.6cm}
\includegraphics[scale=0.8,height=5cm,width=14cm]{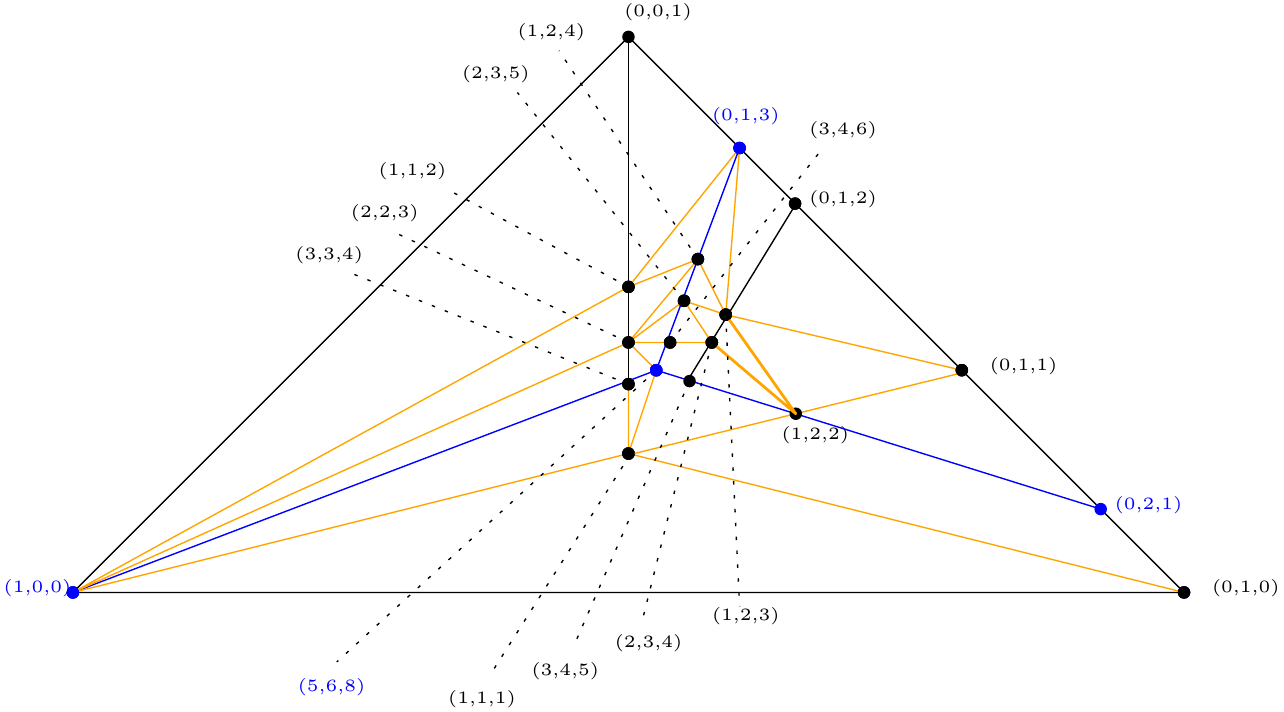}
\caption{An embedded resolution of $E_{7,0}$}
\end{figure}

\section{Jet Schemes and Toric Embedded Resolution of  $E_{0,7}$}

\noindent  The singularity of $X\subset \mathbb{C}^3$ defined by  the equation:
  $$z^{3}+y^{5}+x^{2}y^{2}=0$$ 
is called $E_{0,7}$-type singularity. The singular locus is $\{y=z=0\}$. The jet graph representing the irreducible jet schemes is obtained as:

\begin{figure}[H]
	\setlength{\unitlength}{0.05cm}
\includegraphics[scale=0.8,height=18cm,width=17cm]{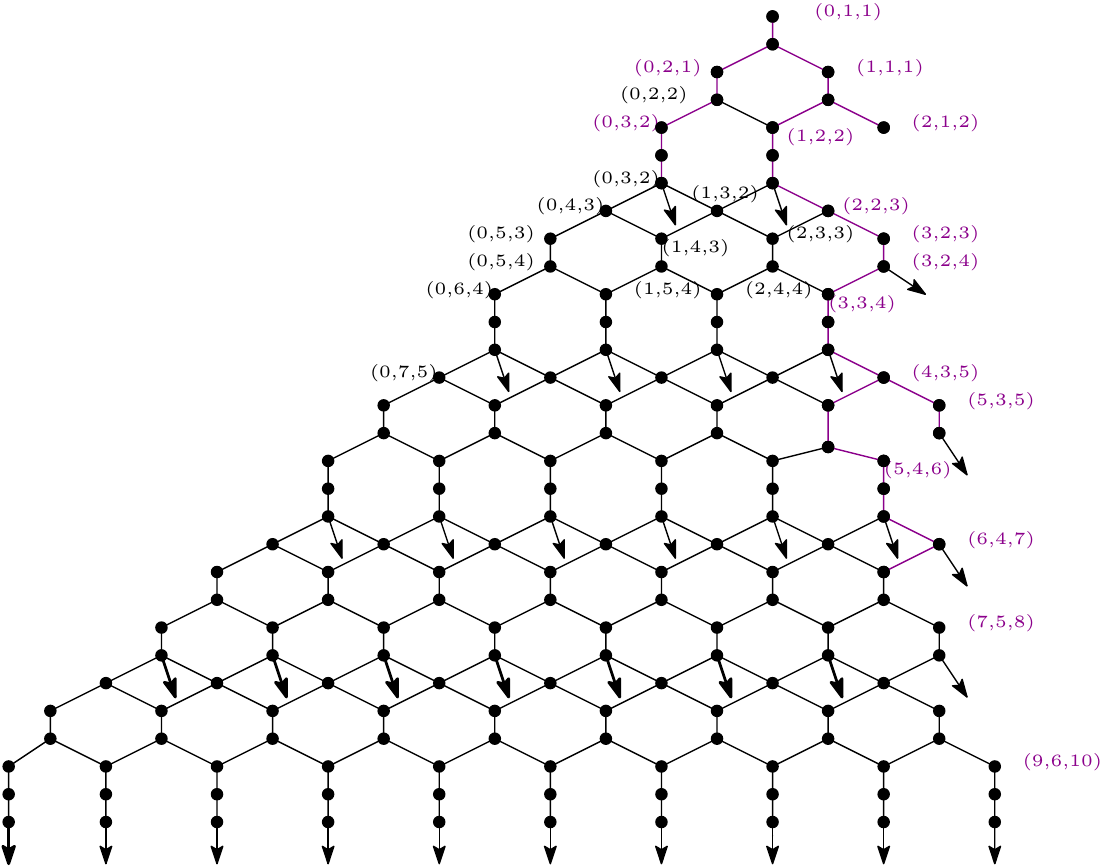}
	\caption{ Jet schemes of $E_{0,7}$}
\end{figure}

\begin{thm}
Let $X$ be a surface of type $E_{0,7}$. The monomial valuations associated with the vectors  
$\{(0,1,1),(0,2,1),(1,1,1)$, $(0,3,2)$, $(1,1,2)$, $(1,2,2)$, $(2,1,2),$ $(2,2,3),$ $(3,2,3),$ $(3,2,4),$ $(3,3,4),$ $(4,3,5),$ $(5,3,5))$
$(5,4,6),(6,4,7),(7,5,8),(9,6,10)\}$ belong to $EV(X).$ There exists a toric birational map $\mu_\Sigma:Z_\Sigma\longrightarrow \mathbb{C}^3$ which is an embedded resolution of $X\subset \mathbb{C}^3$ such that the irreducible components of the exceptional divisor of $\mu_\Sigma$ correspond to the irreducible components of the $m$-th jet schemes of $X$ (centered at the singular locus and the intersection of $X$ with the coordinate hyperplane). Moreover this yields a construction (not canonical) of $\mu_\Sigma.$    
\end{thm}

\begin{figure}[H]
\setlength{\unitlength}{0.05cm}
\includegraphics[scale=0.8,height=6cm,width=14cm]{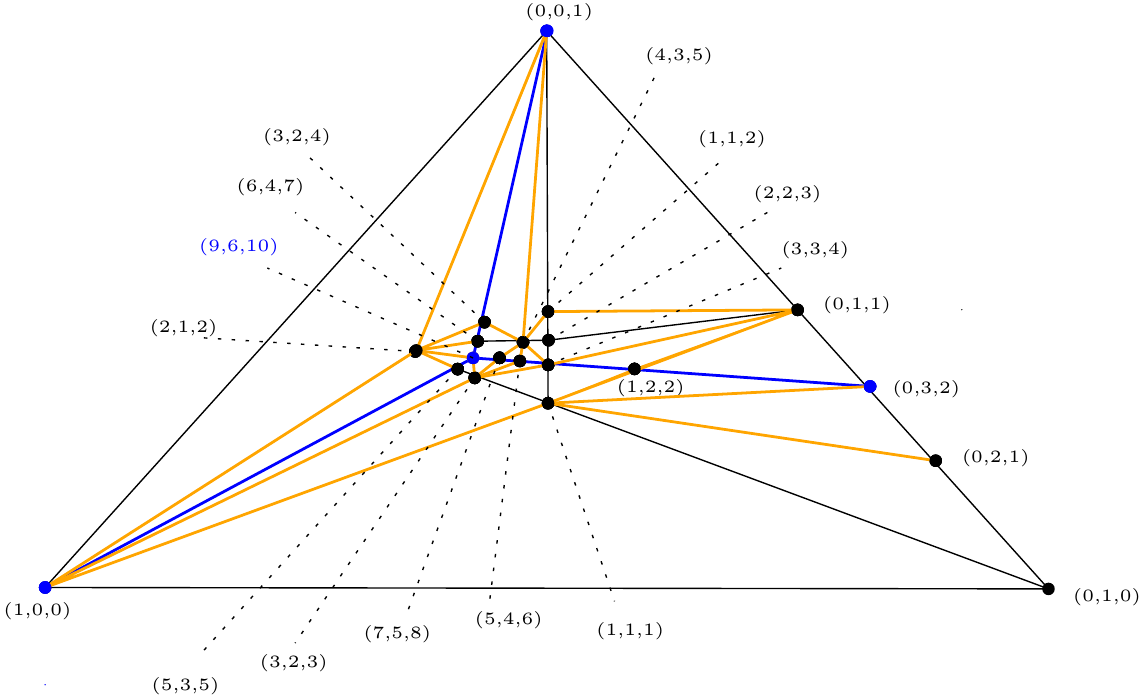}
\caption{An embedded resolution of $E_{0,7}$ }
\end{figure}

\section{Jet Schemes and Toric Embedded Resolution of $F_{k-1}$}

\noindent  The singularity of $X\subset \mathbb{C}^3$ defined by  the equation:
$$z^3+(x+y^{2k})z^2+2xy^{2k}z+(x^2+y^{3})y^{2k}=0$$
is called $F_{k-1}$-type singularity. The singular locus is $\{y=z=0\}$. 

\begin{thm}
Let $X$ be a surface singularity of type $F_{k-1}$. The monomial valuations associated with the vectors:
\begin{itemize}
\item $ (0,1,1),\dots, (0,1,k)$
\item $(1,1,1),\dots,(1,1,k+1)$
\item $(2,1,2),\dots,(2,1,k+1)$
\item $(3,1,3),\dots,(3,1,k)$
\item $\dots$
\item $(a,1,b)$
\item $(2,2,2k+1),(3,2,2k+2),(4,2,2k+1),(6,2,2k) \dots (c,2,d)$
\item $(4,3,3k+2),(5,3,3k+2),(7,3,3k+1),(9,3,3k) \dots {(2k+3,3,2k+3)}$
\item {$(3k+2,3,3k+2)$} if $k=3m+1$
\end{itemize}
\noindent with 
\begin{itemize}
\item $(a,1,b)=(\frac{2k+3}{3},1,\frac{2k+3}{3})$ and $(c,2,d)=(\frac{4k+6}{3},2,\frac{4k+6}{3})$ if $k=3m$ for $m\in \mathbb N$; 
\item $(a,1,b)=(\frac{2k+1}{3},1,\frac{2k+4}{3})$ and $(c,2,d)=(\frac{4k+2}{3},2,\frac{4k+8}{3})$ if $k=3m+1$ for $m\in \mathbb N$
\item  $(a,1,b)=(\frac{2k-1}{3},1,\frac{2k+5}{3})$ and $(c,2,d)=(\frac{4k+4}{3},2,\frac{4k+7}{3})$ if $k=3m+2$ for $m\in \mathbb N*$
\end{itemize} 
belong to $EV(X)$. Moreover there exists a  birational map $\mu_\Sigma:Z_\Sigma\longrightarrow \mathbb{C}^3$ which is an embedded resolution of $X\subset \mathbb{C}^3$ such that the irreducible components of the exceptional divisor of $\mu_\Sigma$ correspond to the irreducible components of the $m$-th jet schemes of $X$ (centered at the singular locus and the intersection of $X$ with the coordinate hyperplane). Moreover this yields a construction (not canonical) of $\mu_\Sigma.$    
\end{thm}

\noindent The jet graph representing the irreducible components of the jet schemes projecting on the singular locus is given by:

\begin{figure}[H]
\setlength{\unitlength}{0.6cm}
\includegraphics[scale=0.6,height=17cm,width=14cm]{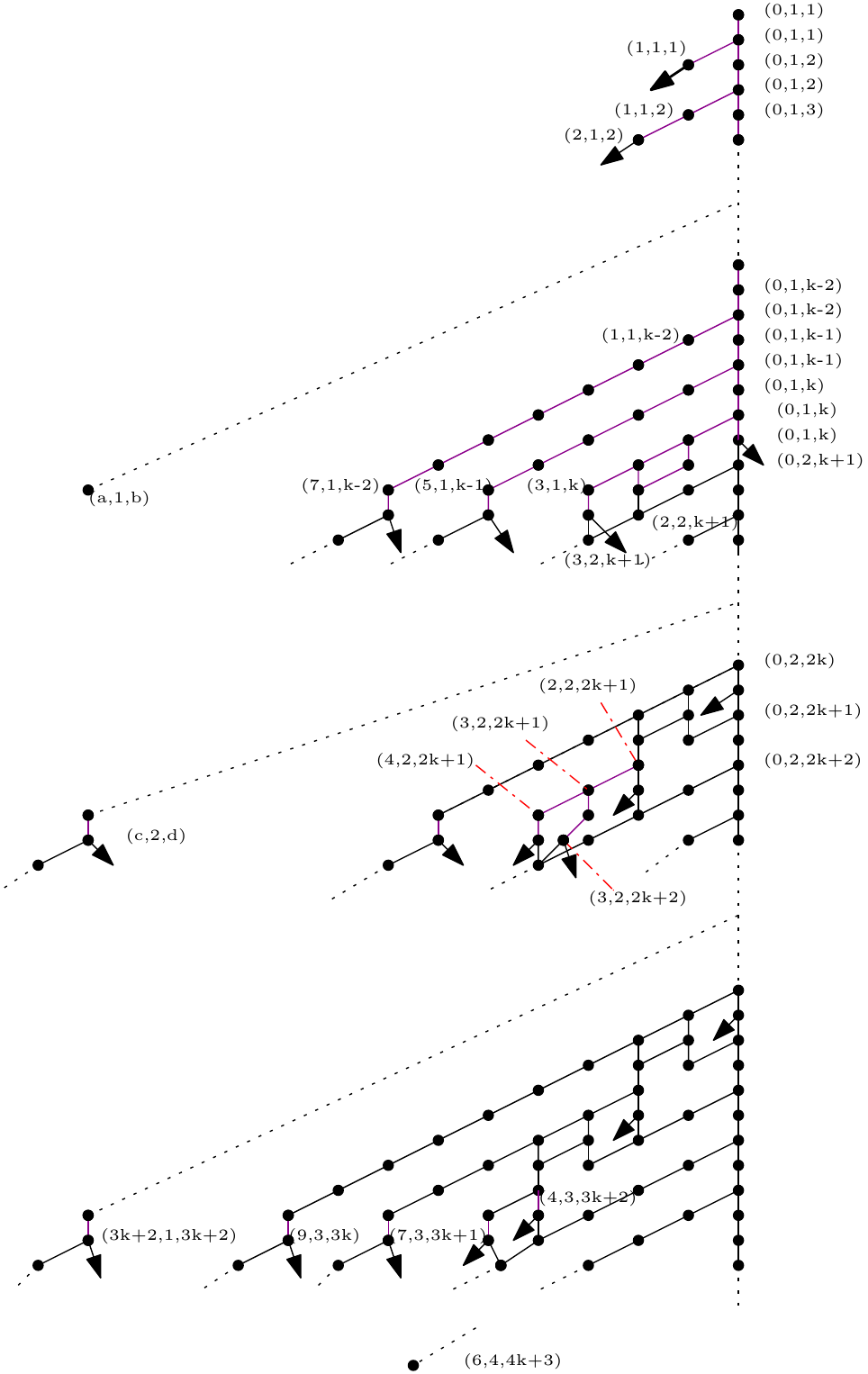}
\caption{ Jet schemes of $F_{k-1}$ }
\end{figure}

\begin{figure}[H]
\setlength{\unitlength}{0.05cm}
\includegraphics[scale=0.6,height=7cm,width=14cm]{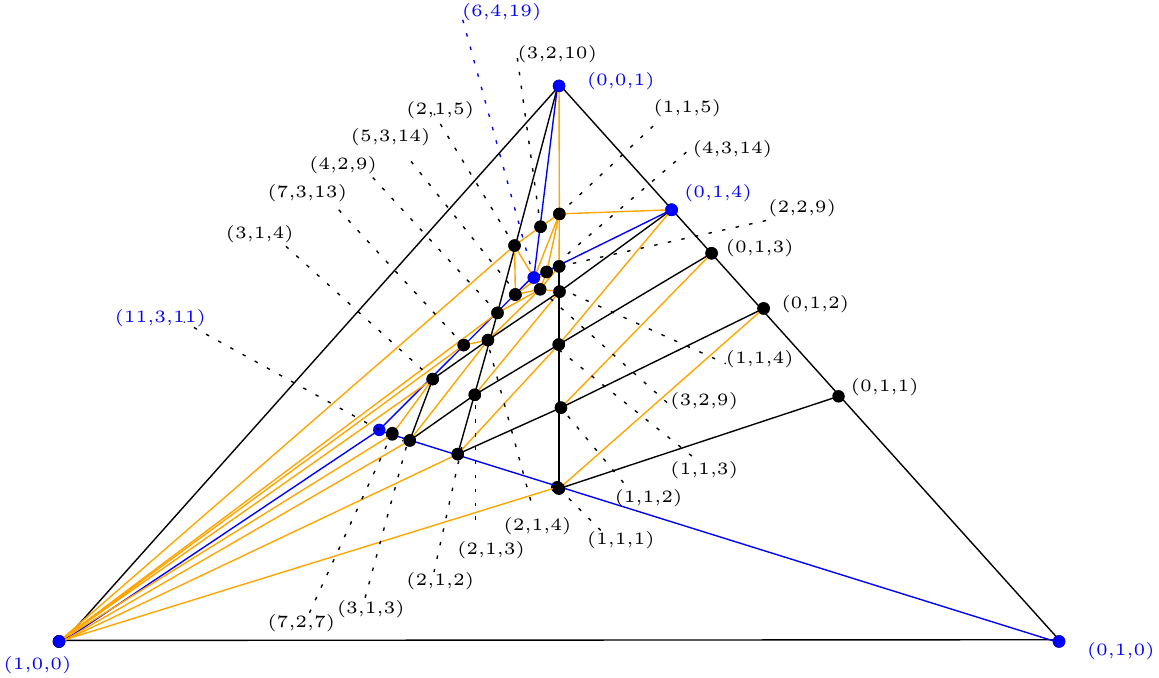}
\caption{{An embedded resolution of $F_3$} }
\end{figure}

\section{Jet Schemes and Toric Embedded Resolution of $H_{n}$}

\noindent The singularity of $X\subset \mathbb{C}^3$ defined by  the equation: 
 \begin{itemize}
\item $z^{3}+x^{2}y(x+y^{k-1})=0$ where $n=3k-1$
\item $z^{3}+xy^{k}z+x^{3}y=0$ where $n=3k$
\item $z^{3}+xy^{k+1}z+x^3y^2=0$ where $n=3k+1$
\end{itemize}
is called $H_n$-type singularity. 
\begin{thm} Let $X$ be a surface of type $H_n$. The monomial valuations associated with the vectors:
\begin{enumerate}
\item $n=3k-1$
\begin{itemize}
\item { $(2,0,1),(3,0,2)$}
\item $(0,1,1), (1,1,2),\dots (k-1,1,k)$
\item $(0,2,1),(1,2,2), \dots (2k-2,2,2k-1)$
\item $(0,3,1),(1,3,2) , \dots (3k-3,3,3k-2)$
\item$(1,0,1),(1,1,1),(2,1,2),\dots (k,1,k)$
\end{itemize}
\item $n=3k$
\begin{itemize}
\item { $(2,0,1))$}
\item $(0,1,1), (1,1,2),\dots (k,1,k+1)$
\item $(0,2,1),(1,2,2), \dots (2k-1,2,2k)$
\item $(0,3,1),(1,3,2) , \dots (3k-2,3,3k-1)$
\item$(1,0,1),(1,1,1),(2,1,2),\dots (k,1,k)$
\end{itemize}
\item $n=3k-1$
\begin{itemize}
\item $(0,1,1), (1,1,2),\dots (k,1,k+1)$
\item $(0,2,1),(1,2,2), \dots (2k,2,2k+1)$
\item $(0,3,2),(1,3,3) , \dots (3k-1,3,3k+1)$
\item$(1,0,1),(1,0,2),(2,0,1)$
\item $(1,1,1),(2,1,2),\dots (k,1,k)$
\end{itemize}
\end{enumerate}
belong to $EV(X) .$ Moreover there exists a toric birational map $\mu_\Sigma:Z_\Sigma\longrightarrow \mathbb{C}^3$ which is an embedded resolution of $X\subset \mathbb{C}^3$ such that the irreducible components of the exceptional divisor of $\mu_\Sigma$ correspond to the irreducible components of the $m$-th jet schemes of $X$ (centered at the singular locus and the intersection of $X$ with the coordinate hyperplane). This yields a construction (not canonical) of $\mu_\Sigma.$    

\end{thm}
\noindent The tree representing the irreducible components of the jets schemes projecting on the singular locus $\{x=z=0\}$ and the axe $\{y=z=0\}$ included in $X$ is the following:

\begin{figure}[H]
\setlength{\unitlength}{0.4cm}
\includegraphics[scale=1,height=18cm,width=18cm]{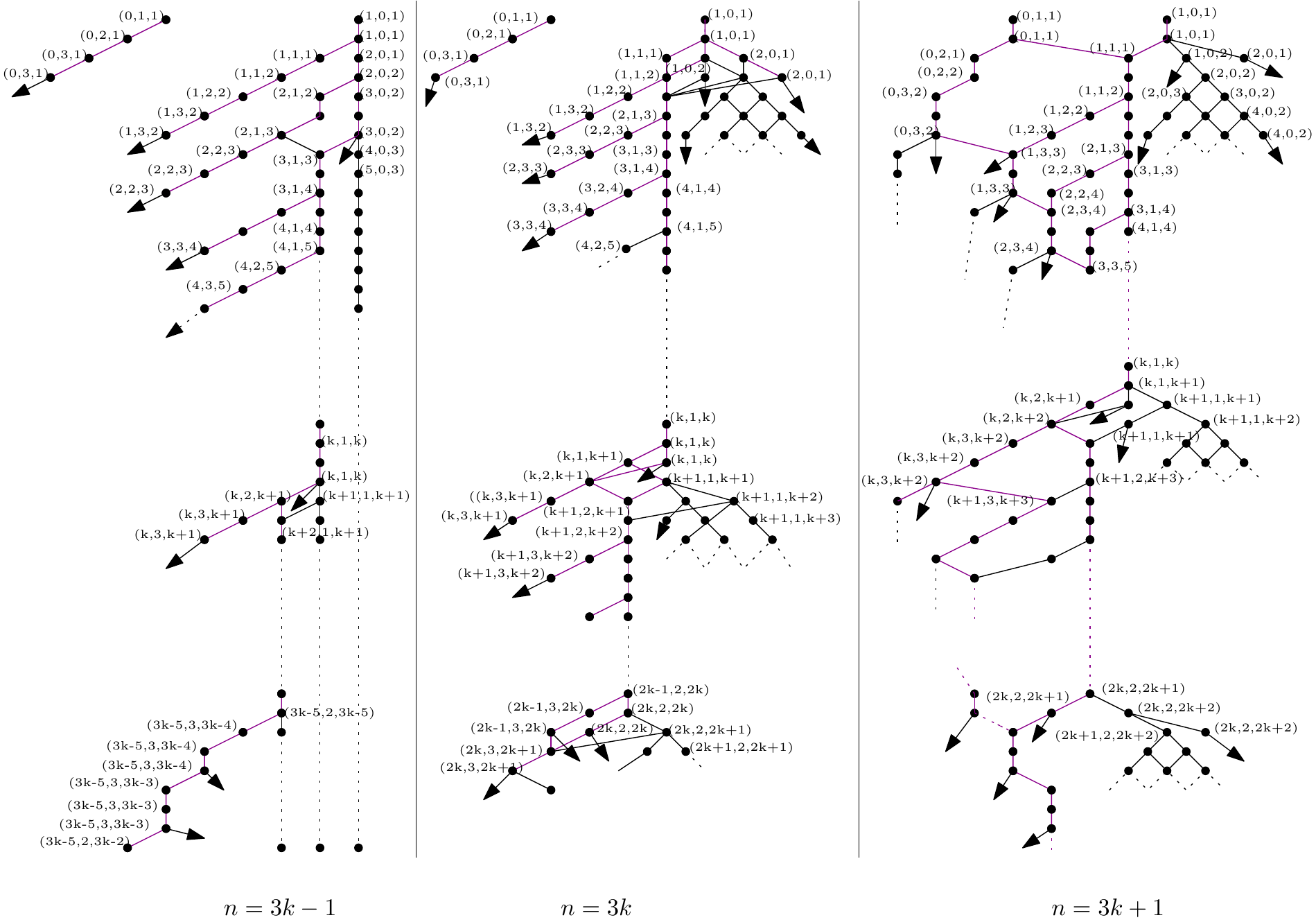}
\caption{ Jets schemes of $H_n$ }
\end{figure}

An embedded resolution for each case is represented on the figure below:

\begin{figure}[H]
\includegraphics[scale=1,height=21cm,width=18cm]{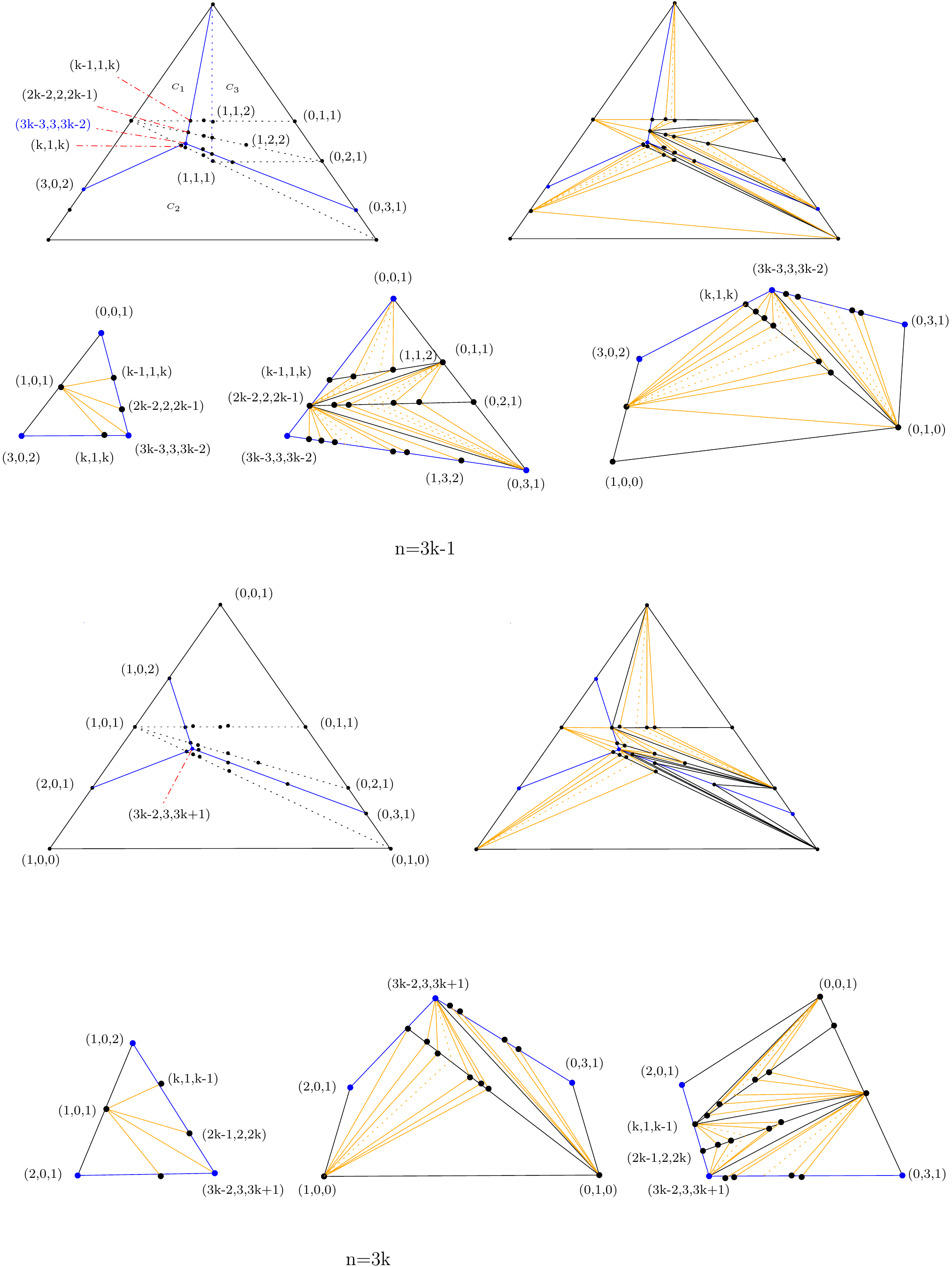}
\caption{ Embedded resolutions of $H_n$ }
\end{figure}

\begin{figure}[H]
\includegraphics[scale=1,height=10cm,width=18cm]{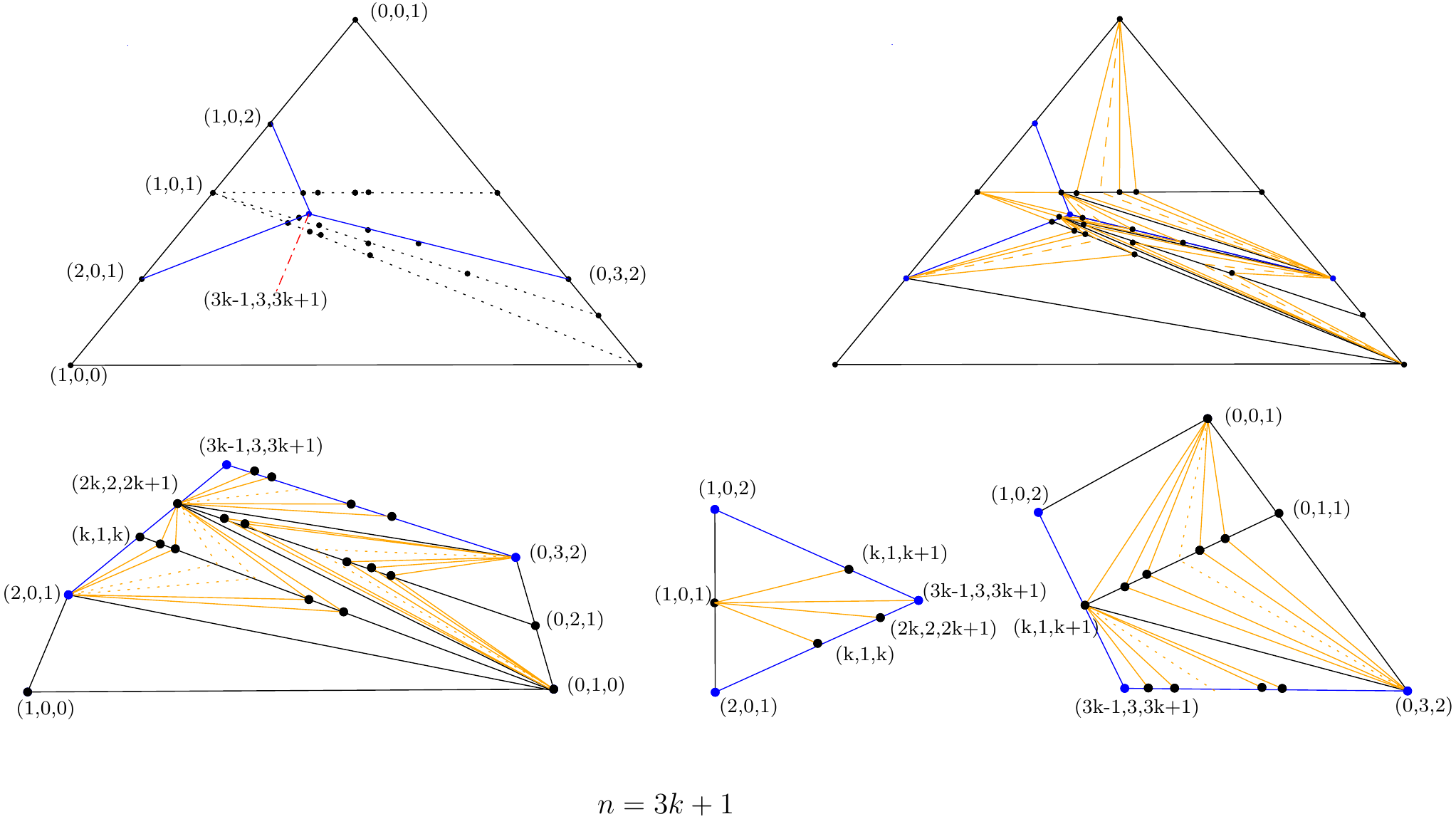}
\caption{ Embedded resolutions of $H_{3k+1}$ }
\end{figure}


\end{document}